\documentclass[12pt,fleqn]{article}

\usepackage{amscd,amsmath,amssymb,amsthm}
\usepackage{geometry} 
\usepackage[pdfpagemode=UseNone,pdfstartview=FitH]{hyperref}
\usepackage{indentfirst}

\newcommand{\A}{{\mathcal A}}
\newcommand{\Ap}[1][]{A_p #1}

\newcommand{\Bp}[1][]{B_p #1}
\newcommand{\bdry}[1]{\partial #1}
\newcommand{\closure}[1]{\overline{#1}}
\newcommand{\comp}{\circ}

\newcommand{\dint}{\ds{\int}}

\newcommand{\ds}[1]{\displaystyle #1}
\newcommand{\dualp}[3][]{\left(#2,#3\right)_{#1}}

\newcommand{\F}{{\mathcal F}}

\newcommand{\incl}{\hookrightarrow}
\newcommand{\M}{{\mathcal M}}
\newcommand{\N}{\mathbb N}
\newcommand{\norm}[2][]{\left\|#2\right\|_{#1}}
\renewcommand{\o}{\text{o}}
\newcommand{\PS}[1]{$(\text{PS})_{#1}$}
\newcommand{\pnorm}[2][]{\if #1'' \left|#2\right|_p \else \left|#2\right|_{#1} \fi}
\newcommand{\R}{\mathbb R}
\newcommand{\RP}{\R \text{P}}

\newcommand{\seq}[1]{\left(#1\right)}
\newcommand{\set}[1]{\left\{#1\right\}}

\newcommand{\Z}{\mathbb Z}
\newcommand{\om}{\int_{\Omega}}
\newcommand{\w}[1]{\widetilde{#1}}
\newcommand{\e}{e_{\varepsilon,\delta}}

\DeclareMathOperator{\divg}{div}

\newenvironment{enumroman}{\begin{enumerate}
		\renewcommand{\theenumi}{$(\roman{enumi})$}
		\renewcommand{\labelenumi}{$(\roman{enumi})$}}{\end{enumerate}}

\newenvironment{properties}[1]{\begin{enumerate}
		\renewcommand{\theenumi}{$(#1_\arabic{enumi})$}
		\renewcommand{\labelenumi}{$(#1_\arabic{enumi})$}}{\end{enumerate}}

\newtheorem{corollary}{Corollary}[section]
\newtheorem{lemma}[corollary]{Lemma}
\newtheorem{proposition}[corollary]{Proposition}
\newtheorem{theorem}[corollary]{Theorem}

\theoremstyle{definition}

\theoremstyle{remark}

\numberwithin{equation}{section}

\title{\bf Critical quasilinear elliptic systems with convex subcritical perturbations}
\author{\bf Elisandra Gloss \& Bruno Ribeiro\\
	Departamento de Matem\'atica\\
	Universidade Federal da Para\'iba\\
	Cidade Universit\'aria, 58051-900, João Pessoa, Brazil\\
	\em elisandra.gloss@academico.ufpb.br \& bhcr@academico.ufpb.br\\
	[\bigskipamount]
	\bf Artur Jorge Marinho \& Kanishka Perera\\
	Department of Mathematics\\
	Florida Institute of Technology\\
	150 W University Blvd, Melbourne, FL 32901-6975, USA\\
	\em amarinho2024@my.fit.edu \& kperera@fit.edu}
\date{}

\begin{document}
	
	\maketitle
	
	\begin{abstract}
		We study a class of Dirichlet problems for coupled quasilinear elliptic systems driven by the $p$--Laplacian in a bounded domain, where the nonlinear terms split into a critical homogeneous part and a convex subcritical perturbation. Using variational methods, a nonlinear eigenvalue theory based on the Fadell--Rabinowitz $\mathbb Z_2$--cohomological index, and a refined linking construction, we prove the existence of a nontrivial solution for every parameter value under suitable dimensional assumptions. We also treat a nonhomogeneous version with small forcing terms by applying a recent abstract perturbation theorem, yielding the existence of two distinct nontrivial solutions. 
		
	\end{abstract}

	%\begin{center}
	%\begin{minipage}{12cm}
	%\tableofcontents
	%\end{minipage}
	%\end{center}

	\section{Introduction}

	Quasilinear elliptic problems involving critical nonlinearities have been a central topic in nonlinear analysis over the past decades.
	The presence of the critical Sobolev exponent marks the threshold at which compactness of the Sobolev embedding fails, leading to a rich and delicate existence theory that combines variational methods, concentration--compactness arguments, and topological tools.
	Since the pioneering works of Brezis and Nirenberg \cite{Brezis-Nirenberg-1983}, Cerami, Fortunato and Struwe \cite{Cerami-Fortunato-Struwe-1984}, Gueda and Veron \cite{Guedda-Veron-1989}, Tarantello \cite{Tarantello-1992}, Azorero and Alonso \cite{Garcia-Azorero-Peral1994} and many others, critical growth problems  have been extensively investigated, both in homogeneous and nonhomogeneous settings.
	
	Elliptic systems with critical nonlinearities present additional difficulties and new phenomena that are not present in the scalar case.
	Besides the loss of compactness at the critical level, one must deal with genuine coupling effects, the possible existence of semitrivial solutions, and the lack of a natural ordering or maximum principle.
	As a consequence, the variational structure of systems is typically more involved, and the construction of critical points requires refined minimax techniques.
	
	For elliptic systems with critical growth, there is by now an extensive literature. In the semilinear case ($p=2$), one can refer, for instance, to
	\cite{Alves-DeMoraisFilho-Souto-2000,Amster-Napoli-Mariani-2002, Chen-Zou-CV-2015, DeMoarisFilho-Faria-Miyagaki-Pereira-2012, Furtado-daSilva2012} for critical Sobolev-type systems, and to nonhomogeneous critical frameworks developed  in \cite{DeMoraisFilho-Pereira-2008,Ribeiro-2010} and related  critical-growth models in \cite{Gloss-Medeiros-Severo-2023}. In the quasilinear setting (general $p$), systems driven by
	$p$-Laplacian-type operators were studied, for instance, in \cite{Boccardo-DeFigueiredo-2002}, while critical $p$--Laplacian systems and their variational structures are treated in
	\cite{Barbosa-Montenegro-2011, Bisci&all-2025, Carmona&all-2013, Chu-Lei-Suo-2017, Chu-Tang-2013,
		DeMoraisFilho-Souto-99, Guo-Perera-Zou-2017,Liu-Zhao-Liu-2018, Lu2012, Silva-Macedo-2018}, while further developments for quasilinear systems of gradient or anisotropic type can be found in \cite{Candela&all-2021,Gluck-2025,Marinho-Perera-2025}.

	In this paper we study critical quasilinear elliptic systems of the form
	\begin{equation}\label{101}
		\left\{
		\begin{aligned}
			- \Delta_p u & = \lambda F_u(u,v) + Q_u(u,v) && \text{in } \Omega,\\
			- \Delta_p v & = \lambda F_v(u,v) + Q_v(u,v) && \text{in } \Omega,\\
			u = v & = 0 && \text{on } \partial\Omega,
		\end{aligned}
		\right.
	\end{equation}
	where $\Omega$ is a bounded domain in $\R^N$, $\Delta_p u = \divg\, (|\nabla u|^{p-2}\, \nabla u)$ is the $p$-Laplacian of $u$, with $1 < p < N$, $\lambda > 0$ is a parameter and the nonlinearity $Q$ exhibits a \emph{critical--subcritical decomposition}
	\[
	Q = H + G.
	\]
	The function $F$ is assumed to be convex, even, and $p$--homogeneous, while $H$ is $p^*$--homogeneous and encodes the critical growth of the system. 
	The lower--order term $G$ is subcritical and satisfies an Ambrosetti--Rabinowitz type condition.
	This structure is natural from the variational viewpoint and allows us to isolate the critical effects associated with the Sobolev exponent from the compact perturbations introduced by $G$. The precise hypotheses for $F$ and $G$ are described later.
	
	When $G\equiv0$ in system \eqref{101}, de Morais Filho and Souto \cite{DeMoraisFilho-Souto-99} proved the existence of nontrivial solutions for such systems under dimensional restrictions and for $\lambda>0$ sufficiently small.
	This result can be seen as a quasilinear counterpart of early scalar results, where compactness is recovered either by restricting parameters or by working below the critical energy level.
	A natural question left open by this and related works is whether nontrivial solutions exist for \emph{all} $\lambda>0$, and how the answer depends on the dimension and on the structure of the nonlinearities.
	
	Our first goal is to address this question for the homogeneous problem and for systems with  subcritical perturbations.
	We prove that, under suitable dimensional conditions, system \eqref{101} admits a nontrivial solution for every $\lambda>0$.
	Moreover, we construct a sequence of variational eigenvalues associated with the problem by means of the $\mathbb Z_2$--cohomological index of Fadell and Rabinowitz \cite{Fadell-Rabinowitz-1978} (see Theorem \ref{Theorem 4}) and, when the parameter $\lambda$ does not belong to this eigenvalue sequence, the dimensional restrictions in our results can be substantially relaxed. This strategy has proved to be particularly effective in critical problems; see, for instance, \cite{Guo-Perera-Zou-2017,Perera-CV-2021,Pe-Ag-OR-book}.
	
	A second goal of this paper is to study nonhomogeneous perturbations of \eqref{101}.
	More precisely, we consider the system
	\begin{equation}\label{101 pert}
		\left\{
		\begin{aligned}
			- \Delta_p u & = \lambda F_u(u,v) + Q_u(u,v) + \gamma(x) && \text{in } \Omega,\\
			- \Delta_p v & = \lambda F_v(u,v) + Q_v(u,v) + \eta(x) && \text{in } \Omega,\\
			u = v & = 0 && \text{on } \partial\Omega,
		\end{aligned}
		\right.
	\end{equation}
	where $\gamma,\eta$ belong to the natural dual space $L^{(p^*)\prime}(\Omega)$, with $p^*=Np/(N-p)$.
	Problems of this type go back to the seminal work of Tarantello \cite{Tarantello-1992} in the scalar case and have been extensively studied since then.
	In the critical regime, the presence of nonhomogeneous terms typically  requires perturbative arguments.
	
	To treat \eqref{101 pert}, we adopt the abstract perturbation framework developed recently by Perera \cite{Perera-JDE-2024}.
	Roughly speaking, this theory shows that, under suitable linking and compactness assumptions for the homogeneous functional, small nonhomogeneous perturbations give rise to multiple nontrivial solutions.
	In our setting, we verify these assumptions by combining the eigenvalue structure of the homogeneous system, sharp Sobolev-type inequalities associated with the critical nonlinearity $H$, and a careful construction of linking paths using truncated extremal functions.
	
	Finally, we address the issue of semitrivial solutions.
	Since systems of the form \eqref{101} may admit solutions with one vanishing component, we introduce an additional structural condition on the nonlinearities that excludes this possibility.
	Under this assumption, all solutions obtained in our main theorems are fully nontrivial, reflecting the genuine coupling of the system.
	
	The paper is organized as follows.
	We finish this introduction by stating the main theorems of this paper; then, in Section~\ref{sec-prelim}, we collect preliminary material, including basic properties of homogeneous functions and the abstract critical point framework based on the cohomological index.
	Section~\ref{sec-PS} is devoted to the verification of a local Palais--Smale condition below the critical energy level.
	In Section~\ref{sec-geo} we establish the geometric conditions required for the linking argument, finishing the proof of the main theorems.

	\subsection{Hypotheses and main theorems}

	Assume $F: \R^2 \to \R$ is a $C^1$-function satisfying
	\begin{properties}{F}
		\item \label{F1} $F$ is $p$-homogeneous and even: $F(\alpha s,\alpha t) = |\alpha|^p\, F(s,t)$ for all $(s,t) \in \R^2$ and $\alpha \in \R$;
		\item \label{F2} $F$ is uniformly positive: $\exists c_F > 0$ such that $F(s,t) \ge c_F\, (|s|^p + |t|^p)$ for all $(s,t) \in \R^2$,
		\item \label{F3} $F$ is convex.
	\end{properties}
	Let \(Q=H+G\) for $C^1$ functions $H,G: \R^2\to[0,\infty)$ satisfying
	\begin{properties}{H}
		\item \label{H1} $H$ is $p^*$-homogeneous: $H(\alpha s,\alpha t) = \alpha^{p^*}\, H(s,t)$ for all $(s,t) \in \R^2$ and $\alpha>0$;
		\item \label{H2} $\exists\, c_H > 0$ such that $H(|s|,|t|)\ge H(s,t) \ge c_H\, (|s|^{p^*} + |t|^{p^*})$ for all $(s,t) \in \R^2$;
		\item \label{H3} $(s,t)\mapsto H(s^{1/p^*},t^{1/p^*})$ is concave on $[0,\infty)\times[0,\infty)$;
	\end{properties}
	and 
	\begin{properties}{G}
		\item \label{G subcri} $G$ is subcritical: There are $C>0$ and $p_i\in(p,p^\ast)$, $i\in\{1,2,3,4\}$, such that 
		$$|\nabla G( s, t)| \leq C(|s|^{p_1-1}+|s|^{p_2-1}+|t|^{p_3-1}+|t|^{p_4-1})\quad\forall (s,t) \in \R^2;$$
		\item \label{AR G} $\exists r\in(p,p^\ast)$ such that $rG(s,t)\leq  sG_s(s,t)+tG_t(s,t) $ for all $(s,t) \in \R^2$; 
	\end{properties}
	For a $q$-homogeneous function $K\in C^1(\R^2,\R)$ we know that
	\begin{equation} \label{201}
		s\, K_s(s,t) + t\, K_t(s,t) = q\, K(s,t) \quad \forall (s,t) \in \R^2
	\end{equation}
	and that $\exists C_K > 0$ such that
	\begin{equation} \label{202}
		|K(s,t)|\le C_K\, (|s|^q + |t|^q) \quad \forall (s,t) \in \R^2.
	\end{equation}

	Our first main result reads as follows.
	
	\begin{theorem}\label{Theorem 1}
		Assume \ref{F1}--\ref{F3}, \ref{H1}--\ref{H3}, and \ref{G subcri}--\ref{AR G}.  
		If
		\(
		{N^{2}}/{(N+1)} > p^{2},
		\)
		then system~\eqref{101} admits a nontrivial solution for every $\lambda>0$. Moreover, if
		\(1<p<N\) and ${N^{2}}/{(N+1)} \le p^{2}$, 
		the same conclusion holds provided that the nonlinearity $G$ satisfies, in addition,
		\begin{properties}{G}
			\addtocounter{enumi}{2}
			\item\label{G small N}
			There exist constants $c_G>0$ and
			\(
			q>p^{*}-{Np}/{[N(p-1)+p]}
			\)
			%$\max\set{p^{*}-{Np}/{[N(p-1)+p]},p^*(1-1/p)}=p^{*}-{Np}/{[N(p-1)+p]}$ if $N\le p^2+p$, which occurs for  \(1<p<N\) and ${N^{2}}/{(N+1)} \le p^{2}$.
			such that 
			\[
			G(s,t)\ge c_G \left(|s|^{q}+|t|^q\right) \quad \text{for all }(s,t)\in\R^{2}.
			\]
		\end{properties}
	\end{theorem}
	
	We also show that the restrictions on the relations between the parameters $N$, $p$, and $q$ can be relaxed whenever $\lambda$ does not belong to a sequence of minimax eigenvalues $\{\lambda_k\}$ associated with the problem 
	\begin{equation}\label{102}
		\left\{
		\begin{aligned}
			- \Delta_p u & = \lambda F_u(u,v) && \text{in } \Omega,\\[4pt]
			- \Delta_p v & = \lambda F_v(u,v) && \text{in } \Omega,\\[4pt]
			u = v & = 0 && \text{on } \partial\Omega,
		\end{aligned}
		\right.
	\end{equation}
	which is constructed via the $\Z_2$--cohomological index (see Theorem~\ref{Theorem 4}).
	
	Eigenvalue problems for quasilinear elliptic systems go back, in particular,
	to the work of de Th\'elin \cite{DeThelin-1990}; see also
	\cite{Arcoya-Borgia-Cingolani-2026, Borgia-Cingolani-Vannella-2024}
	for recent results concerning resonance phenomena for such systems.
	
	The assumptions in the next theorem are more natural and coincide with those typically imposed in the existing literature on scalar equations as well as on gradient systems (see \cite{DeMoraisFilho-Souto-99}).
	
	\begin{theorem}\label{Theorem 2}
		Assume \ref{F1}--\ref{F3}, \ref{H1}--\ref{H3}, and \ref{G subcri}--\ref{AR G}.  
		If $N\ge p^{2}$, then system~\eqref{101} admits a nontrivial solution for every $\lambda\in(0,+\infty)\backslash\{\lambda_k\}$. Moreover, if
		$p<N<p^2$, the same conclusion holds provided that the nonlinearity $G$ satisfies, in addition,
		\begin{properties}{G}
			\addtocounter{enumi}{3}
			\item\label{G small N nonresonant}
			There exist constants $c_G>0$ and
			\(q>p^{*}-p/(p-1)\)  such that
			\[
			G(s,t)\ge c_G \left(|s|^{q}+|t|^q\right) \quad \text{for all }(s,t)\in\R^{2}.
			\]
		\end{properties}
	\end{theorem}
	
	It is worth noting that the threshold values for the exponent \(q\) appearing in conditions~\ref{G small N} (in the general case) and~\ref{G small N nonresonant} (in the nonresonant case) coincide with those obtained by Arioli and Gazzola in~\cite{Arioli-Gazzola98}. In that work, the authors study a scalar critical \(p\)-Laplacian equation perturbed by a \(p\)-superlinear term. In order to treat
	low-dimensional regimes, they impose a growth condition analogous to \ref{G small N nonresonant} when \(0<\lambda<\lambda_{1}\), and a condition analogous to~\ref{G small N} in the range \(\lambda_{1}\le \lambda<\lambda_{2}\).
	
	\medskip
	
	The conclusions of Theorems~\ref{Theorem 1} and~\ref{Theorem 2} do not rule out the possibility that the nontrivial solution obtained is \emph{semitrivial}, that is, that one of the components vanishes identically, $u \equiv 0$ or $v \equiv 0$, in which case system~\eqref{101} reduces to a single quasilinear equation. In order to exclude this possibility and guarantee the genuine coupling of the system, we impose the following additional structural assumption on the nonlinearities.
	
	\begin{corollary}
		Under the assumptions of Theorem~\ref{Theorem 1} or Theorem~\ref{Theorem 2}, suppose in addition that
		\begin{properties}{P}
			\addtocounter{enumi}{0}
			\item\label{F4} $F_s(0,1)> 0, F_t(1,0)> 0$ \ \ and \ \ $Q_s(0,t)\ \!\!t\ge 0,Q_t(s,0)\ \!\!s\ge 0$ \ \ for all \ \ $(s,t)\in\mathbb R^2$.
		\end{properties}
		Then any solution $(u,v)$ provided by those theorems is \emph{fully nontrivial}, in the sense that
		$u \not\equiv 0$ and $v \not\equiv 0$.
	\end{corollary}

	Let us show some examples of admissible nonlinearities $F$, $H$ and $G$.
	
	\medskip
	
	\noindent An example of $F: \R^2 \to \R$ satisfying \ref{F1}, \ref{F2}, \ref{F3} and \ref{F4} is given by
	\[
	F(s,t)=\left(t^2+cst+s^2\right)^{p/2},\quad 0<c<2.
	\]
	
	\noindent The following functions $H:\R^{2}\to[0,\infty)$ satisfy \ref{H1}, \ref{H2}, \ref{H3} (and \ref{F4}, with $Q=H$).
	\[H(s,t)=a|s|^{p^*}+b|t|^{p^*}+\sum_{i=1}^{k}c_i|s|^{\alpha_i}|t|^{\beta_i}\] with $\alpha_i+\beta_i=p^*$, $\alpha_i,\beta_i>1$, $c_i\ge 0$ and $a,b>0$.

	\medskip

	Throughout the paper we work in the product Sobolev space
	\[
	W:=W_0^{1,p}(\Omega)\times W_0^{1,p}(\Omega),
	\]
	endowed with the norm
	\[
	\|w\|:=\left(\int_\Omega
	\left(|\nabla u|^p+|\nabla v|^p\right)\,dx\right)^{1/p},
	\qquad w=(u,v)\in W.
	\]
	
	We prove Theorems \ref{Theorem 1} and \ref{Theorem 2} by associating system \eqref{101} to the $C^1$-functional $E_0:W\to\R$ defined by
	\begin{equation} \label{E0 def}
		E_0(w) = \frac{1}{p} \int_\Omega (|\nabla u|^p + |\nabla v|^p)\, dx - \lambda \int_\Omega F(u,v)\, dx -  \int_\Omega Q(u,v)\, dx
	\end{equation}
	and proving existence of nontrivial critical points by using a general variational scheme given by Yang and Perera \cite{Ya-Pe-2016}. 
	
	\medskip
	
	Now, let us address the nonhomogeneous counterpart of \eqref{101}, by considering the system \eqref{101 pert}
	where $\gamma,\eta\in L^{(p^*)\prime}(\Omega)$. This problem can be treated as a nonhomogeneous perturbation of \eqref{101} and we prove that if the perturbation is small enough there exists two nontrivial solutions for this problem. This result can be compared with the classical work of Tarantello \cite{Tarantello-1992}, which initiated the study of nonhomogeneous elliptic problems involving critical growth and has since inspired a vast literature. For a comprehensive background on related scalar problems and on the perturbative variational framework adopted here, we refer the reader to \cite{Perera-JDE-2024}.
	\begin{theorem}\label{Theorem 5}
		Under the same conditions of Theorem \ref{Theorem 1} or Theorem \ref{Theorem 2}, there exists  $\mu_0>0$ such that if $0<\|\gamma\|_{{(p^*)\prime}}+\|\eta\|_{{(p^*)\prime}}<\mu_0$ then \eqref{101 pert} admits two nontrivial solutions.
	\end{theorem}

	Weak solutions of \eqref{101 pert} coincide with critical points of the $C^1$-functional $E:W\to\R$ defined by
	\begin{equation} \label{103}
		E(w) = \frac{1}{p} \int_\Omega (|\nabla u|^p + |\nabla v|^p)\, dx - \lambda \int_\Omega F(u,v)\, dx -  \int_\Omega Q(u,v)\, dx-\int_\Omega (\gamma u+\eta v)\, dx.
	\end{equation}
	We shall find critical points for $E$ when $\gamma$ and $\eta$ are small functions in  $L^{(p^\ast)'}$, comparing it with the functional $E_0:W\to\R$ defined in \eqref{E0 def}, using a general perturbation method given in \cite{Perera-JDE-2024}.

	\section{Preliminaries}\label{sec-prelim}

	This section collects basic preliminaries used throughout the paper. We recall some elementary properties of homogeneous functions and present the abstract critical point framework that underpins the proofs of our main results.

	\subsection{A characterization of convexity for $p$-homogeneous functions}
	
	Let $1 < p < \infty$ and let $K \in C^1(\R^N,[0,\infty))$ be $p$-homogeneous, i.e.,
	\begin{equation} \label{214}
		K(\alpha y) = \alpha^p\, K(y) \quad \forall y \in \R^N,\, \alpha \ge 0.
	\end{equation}
	Set
	\[
	\beta(y,z) = \nabla K(y) \cdot z, \quad y, z \in \R^N.
	\]
	By \eqref{214},
	\begin{equation} \label{203}
		\beta(y,y) = p\, K(y) \ge 0 \quad \forall y \in \R^N.
	\end{equation}
	
	\begin{proposition} \label{Proposition 1}
		Assume \eqref{214}. Then $K$ is convex if and only if
		\begin{equation} \label{215}
			\beta(y,z) \le \beta(y,y)^{(p-1)/p}\, \beta(z,z)^{1/p} \quad \forall y, z \in \R^N.
		\end{equation}
	\end{proposition}
	
	\begin{proof}
		Suppose $K$ is convex. Then
		\[
		K(\alpha z) \ge K(y) + \nabla K(y) \cdot (\alpha z - y) \quad \forall \alpha > 0.
		\]
		By \eqref{214} and \eqref{203}, this inequality can be written as
		\[
		\beta(y,z) \le \alpha^{-1} \left(1 - \frac{1}{p}\right) \beta(y,y) + \frac{\alpha^{p-1}}{p}\, \beta(z,z).
		\]
		Minimizing the right-hand side over all $\alpha > 0$ gives \eqref{215}. The converse follows by reversing the above steps.
	\end{proof}
	
	\subsection{Eigenvalue problem}
	
	In this section we prove some results for problem \eqref{102}. This eigenvalue problem can be written as
	\begin{equation} \label{218}
		\Ap[w] = \lambda \Bp[w]
	\end{equation}
	in the dual $W^\ast$ of $W$, where $\Ap, \Bp \in C(W,W^\ast)$ are the operators given by
	\[
	\dualp{\Ap[w]}{\xi} = \int_\Omega (|\nabla u|^{p-2}\, \nabla u \cdot \nabla \varphi + |\nabla v|^{p-2}\, \nabla v \cdot \nabla \psi)\, dx
	\]
	and
	\[
	\dualp{\Bp[w]}{\xi} = \int_\Omega (F_u(u,v)\, \varphi + F_v(u,v)\, \psi)\, dx
	\]
	for $w = (u,v)$ and $\xi = (\varphi,\psi)$ in $W$. The operators $\Ap$ and $\Bp$ are $(p - 1)$-homogeneous and odd, i.e.,
	\[
	\Ap[(\alpha w)] = |\alpha|^{p-2}\, \alpha\, \Ap[w], \quad \Bp[(\alpha w)] = |\alpha|^{p-2}\, \alpha\, \Bp[w] \quad \forall w \in W,\, \alpha \in \R.
	\]
	We have
	\[
	\dualp{\Ap[w]}{w} = \norm{w}^p \quad \forall w \in W,
	\]
	and by  \eqref{201} and \ref{F2} it holds
	\[
	\dualp{\Bp[w]}{w} = p \int_\Omega F(u,v)\, dx \ge pc_F \int_\Omega (|u|^p + |v|^p)\, dx > 0 \quad \forall w \in W \setminus \set{0}.
	\]
	By the H\"{o}lder inequalities for integrals and sums,
	\begin{multline*}
		\hspace{-18.83pt} \dualp{\Ap[w]}{\xi} \le \left(\int_\Omega |\nabla u|^p\, dx\right)^{(p-1)/p}\! \left(\int_\Omega |\nabla \varphi|^p\, dx\right)^{1/p} + \left(\int_\Omega |\nabla v|^p\, dx\right)^{(p-1)/p}\! \left(\int_\Omega |\nabla \psi|^p\, dx\right)^{1/p}\\[5pt]
		\le \left(\int_\Omega (|\nabla u|^p + |\nabla v|^p)\, dx\right)^{(p-1)/p}\! \left(\int_\Omega (|\nabla \varphi|^p + |\nabla \psi|^p)\, dx\right)^{1/p} = \norm{w}^{p-1} \norm{\xi}
	\end{multline*}
	for all $w, \xi \in W$, and equality holds if and only if $\tau w = \zeta \xi$ for some $\tau, \zeta \ge 0$, not both zero. Since
	\[
	\dualp{\Bp[w]}{\xi} = \int_\Omega \beta(w,\xi)\, dx,
	\]
	where $\beta(w,\xi) = \nabla F(w) \cdot \xi$, \ref{F3} together with Proposition \ref{Proposition 1} and the H\"{o}lder inequality gives
	\begin{multline*}
		\dualp{\Bp[w]}{\xi} \le \int_\Omega \beta(w,w)^{(p-1)/p}\, \beta(\xi,\xi)^{1/p}\, dx \le \left(\int_\Omega \beta(w,w)\, dx\right)^{(p-1)/p}\! \left(\int_\Omega \beta(\xi,\xi)\, dx\right)^{1/p}\\[5pt]
		= \dualp{\Bp[w]}{w}^{(p-1)/p} \dualp{\Bp[\xi]}{\xi}^{1/p} \quad \forall w, \xi \in W.
	\end{multline*}
	Since $W$ is compactly embedded in $L^p(\Omega) \times L^p(\Omega)$, $\Bp$ is a compact operator. Thus, the eigenvalue problem \eqref{218} fits into the abstract framework considered in Perera et al.\! \cite[Chapter 4]{Pe-Ag-OR-book} %{MR2640827} 
	and Perera \cite{Perera-CV-2021}. %{MR4293883}
	
	Eigenvalues of problem \eqref{102} coincide with critical values of the $C^1$-functional
	\[
	\Psi(w) = \frac{1}{p \dint_\Omega F(u,v)\, dx}, \quad w \in \mathcal{M} = \set{w \in W : \norm{w} = 1}.
	\]
	
	We now recall the definition of cohomological index  as follows. Let $W$ be a Banach space and let $\A$ denote the class of symmetric subsets of $W \setminus \set{0}$. For $A \in \A$, let $\overline{A} = A/\Z_2$ be the quotient space of $A$ with each $w$ and $-w$ identified, let $f : \overline{A} \to \RP^\infty$ be the classifying map of $\overline{A}$, and let $f^\ast : H^\ast(\RP^\infty) \to H^\ast(\overline{A})$ be the induced homomorphism of the Alexander-Spanier cohomology rings. The cohomological index of $A$ is defined by
	\[
	i(A) = \begin{cases}
		\sup \set{m \ge 1 : f^\ast(\omega^{m-1}) \ne 0} & \text{if } A \ne \emptyset\\[5pt]
		0 & \text{if } A = \emptyset,
	\end{cases}
	\]
	where $\omega \in H^1(\RP^\infty)$ is the generator of the polynomial ring $H^\ast(\RP^\infty) = \Z_2[\omega]$. For example, the classifying map of the unit sphere $S^{m-1}$ in $\R^m,\, m \ge 1$ is the inclusion $\RP^{m-1} \incl \RP^\infty$, which induces isomorphisms on $H^q$ for $q \le m - 1$, so $i(S^{m-1}) = m$.
	
	The result below follows from \cite[Theorem 4.6]{Pe-Ag-OR-book}. %{MR2640827}.
	
	\begin{theorem} \label{Theorem 4}
		Let $\F$ denote the class of symmetric subsets of $\mathcal{M}$. For $k \ge 1$, let
		\[
		\F_k = \set{M \in \F : i(M) \ge k},
		\]
		where $i(M)$ denotes the cohomological index of $M$, and set
		\[
		\lambda_k := \inf_{M \in \F_k}\, \sup_{w \in M}\, \Psi(w).
		\]
		Then $\lambda_k \nearrow \infty$ is a sequence of eigenvalues of the eigenvalue problem \eqref{102}.
		\begin{enumroman}
			\item \label{Theorem 4.i} The first eigenvalue is given by $\lambda_1 = \inf \Psi(\mathcal{M})$ and is positive.
			\item \label{Theorem 4.ii} If $\lambda_k = \dotsb = \lambda_{k+m-1} = \lambda$ and $E_\lambda$ denotes the set of eigenfunctions associated with $\lambda$ that lie on $S$, then $i(E_\lambda) \ge m$.
			\item \label{Theorem 4.iii} If $\lambda_k < \lambda < \lambda_{k+1}$, then
			\[
			i(\Psi^{\lambda_k}) = i(\mathcal{M} \setminus \Psi_\lambda) = i(\Psi^\lambda) = i(\mathcal{M} \setminus \Psi_{\lambda_{k+1}}) = k,
			\]
			where $\Psi^a = \set{w \in \mathcal{M} : \Psi(w) \le a}$ and $\Psi_a = \set{w \in \mathcal{M} : \Psi(w) \ge a}$ for $a \in \R$.
		\end{enumroman}
	\end{theorem}
	
	The main result of this section is the following proposition, which extends Theorem 2.3 in Degiovanni and Lancelotti \cite{De-La-JFA-2009} %{MR2514055} 
	to the system \eqref{101}.
	
	\begin{proposition} \label{Proposition 2}
		If $\lambda_k < \lambda_{k+1}$, then the sublevel set $\Psi^{\lambda_k}$ has a compact symmetric subset $\mathcal{C}$ of index $k$ that is bounded in $(L^\infty(\Omega) \times L^\infty(\Omega)) \cap (C^1_\text{loc}(\Omega) \times C^1_\text{loc}(\Omega))$.
	\end{proposition}
	
	\begin{proof}
		Consider the map $K : W \to W,\, \xi = (\varphi,\psi) \mapsto w = (u,v)$, where $w$ is the unique solution of the equation
		\[
		\Ap[w] = \Bp[\xi],
		\]
		i.e., $(u,v)$ is the unique solution of the system
		\begin{equation*} %\label{100}
			\left\{\begin{aligned}
				- \Delta_p u & = F_u(\varphi,\psi) && \text{in } \Omega\\[5pt]
				- \Delta_p v & = F_v(\varphi,\psi) && \text{in } \Omega\\[5pt]
				u = v & = 0 && \text{on } \bdry{\Omega}.
			\end{aligned}\right.
		\end{equation*}
		Set $J = \pi \comp K$, where $\pi(w) = w/\norm{w},\, w \in W \setminus \set{0}$ is the radial projection on $\mathcal{M}$. Denoting by $J^l$ the $l$-fold composition of $J$ with itself, the argument in the proof of Perera \cite[Theorem 1.3]{Perera-CV-2021} %{MR4293883} 
		shows that $\mathcal{C} = \closure{J^l(\Psi^{\lambda_k})}$ is a compact symmetric subset of $\Psi^{\lambda_k}$ with index $k$ for any $l \ge 1$. As in the proof of Degiovanni and Lancelotti \cite[Theorem 2.3]{De-La-JFA-2009} we verify that $\mathcal{C}$ is bounded in $(L^\infty(\Omega) \times L^\infty(\Omega)) \cap (C^1_\text{loc}(\Omega) \times C^1_\text{loc}(\Omega))$ if $l$ is sufficiently large.
	\end{proof}
	
	\subsection{Linking theorems}
	
	As before, we consider $\mathcal{M}$ the unit sphere in $W$,
	\[
	\mathcal{M} = \{ w \in W \setminus \{0\} : \|w\| = 1 \},
	\]
	and the radial projection to $\mathcal{M}$, $\pi: W\setminus\{0\}\to \mathcal{M}$ given by $\pi(w)=w/\|w\|$.
	
	The proofs of Theorems \ref{Theorem 1} and \ref{Theorem 2} are based on the following linking theorem.
	\begin{theorem}[\cite{Ya-Pe-2016}, Theorem 2.2]\label{lk th}
		Let $\Phi$ be a $C^1$-functional defined on $W$. Assume that there exists $c^*\in\mathbb R$ such that $\Phi$ satisfies the $(PS)_c$ condition  for all $c<c^*$.
		Let $A_0$ and $B_0$ be disjoint nonempty closed symmetric
		subsets of $\mathcal{M}$ such that
		\[
		i(A_0) = i(\mathcal{M}\setminus B_0) <\infty.
		\]
		Assume that there exist $R>\rho>0$ and $e\in\mathcal{M}\setminus A_0$ such that
		\begin{eqnarray}\label{lk th2}
			\sup \Phi(A) \leq \inf \Phi(B), \quad \sup \Phi(X) < c^*,
		\end{eqnarray}
		where
		\begin{align*}
			A= \{tw: w\in A_0,\, 0 \leq t\leq R\}\cup \{R\pi((1-t) w+ te) : w\in A_0, \,0 \leq t\leq 1\},\\
			B= \{\rho w: w\in B_0\} \quad\text{and}\quad \quad X= \{tw: w\in A,\,\|w\|= R, \,0 \leq t\leq 1\}.
		\end{align*}
		Let 
		\[\Gamma = \{\gamma \in C(X,W) : \gamma(X) \text{ is closed and } \gamma_{|_A} = id_A\}
		\]
		and set
		\[
		c:=\inf_{\gamma\in\Gamma}\sup_{w\in\gamma(X)}\Phi(w)
		\]
		Then $\inf \Phi(B) \leq c< c^*$ and $c$ is a critical value of $\Phi$.
	\end{theorem}
	
	For the proof of Theorem \ref{Theorem 5} we will apply the result of Perera \cite[Theorem 2.1]{Perera-JDE-2024}.
	At first, let us present some notations. The functional $E_0$ and $E$ defined in \eqref{E0 def} and \eqref{103}, respectively,  can be rewritten as 
	\begin{equation*}
		E_0(w) = I_{p}(w) - \lambda J_{p}(w) - \mathcal{Q}(w),
		\qquad w \in W,
	\end{equation*}
	and
	\begin{equation}\label{eq def general E}
		E(w) = I_{p}(w) - \lambda J_{p}(w) - \mathcal{Q}(w)  - \langle h, w \rangle,
		\qquad w \in W,
	\end{equation}
	where
	\begin{equation*}
		I_{p}(w) = \frac{1}{p} \langle A_{p}w, w \rangle,
		\qquad
		J_{p}(w) = \frac{1}{p} \langle B_{p}w, w \rangle
	\end{equation*}
	are the potentials of \( A_{p} \) and \( B_{p} \), with \( I_{p}(0) = J_{p}(0) = 0 \), and 
	\[
	\mathcal{Q}(w)=\int_\Omega Q(u,v)dx\quad\text{and} \quad\langle h, w \rangle=\int_\Omega\left(\gamma u+\eta v\right)dx,\quad \forall w=(u,v)\in W.
	\]
	As seen in the previous section, $A_p$ satisfies the conditions $(A_1)-(A_2)$ and $B_p$ satisfies the conditions $(B_1)-(B_3)$ from Perera \cite{Perera-JDE-2024}. On the other hand, we can see that $\mathcal{Q}(0) = 0$ and	\( \mathcal{Q}(w) = o(\lVert w\rVert^{p}) \) as \( w \to 0 \); $\mathcal{Q}(w)\geq0$ for all \( w \in W \); and
	\( \mathcal{Q} \) is bounded on bounded subsets of \( W \). 
	
	Assume that there is \( c^{*}_{h} > 0 \) such that \( E \) satisfies the
	\( (\mathrm{PS})_{c} \) condition at all levels \( c < c^{*}_{h} \) and 
	define
	\begin{equation*}
		c^{*} := \liminf_{\| h \|_{*} \to 0} c^{*}_{h}.
	\end{equation*}
	Under these notations and conditions, the following result was proved in \cite[Theorem 2.1]{Perera-JDE-2024}.
	\begin{theorem}\label{thm:2.1} 
		Let \( \lambda_{k} \le \lambda < \lambda_{k+1} \).
		Assume there exist \( R > 0 \) and, for every sufficiently small \( \delta > 0 \),
		a compact symmetric subset \( C_{\delta} \subset \Psi^{\lambda + \delta} \) with
		\( i(C_{\delta}) = k \) and an element \( w_{\delta} \in \mathcal{M} \setminus C_{\delta} \) such
		that, putting
		\[
		A_{\delta} := \left\{
		\pi\bigl((1 - \tau)v + \tau w_{\delta}\bigr)
		: v \in C_{\delta},\; 0 \le \tau \le 1
		\right\},
		\]
		one has
		\begin{align}
			\sup_{u \in A_{\delta}} E_{0}(Ru) &\le 0, \label{eq:2.8}\\
			\sup_{u \in A_{\delta},\,0 \le t \le R} E_{0}(tu) &< c^{*}. \label{eq:2.9}
		\end{align}
		Then there exists \( \mu_{0} > 0 \) such that
		the functional $E$ defined in ~\eqref{eq def general E} possesses two distinct nontrivial critical points
		\( u_{1}, u_{2} \) satisfying
		\begin{equation*}
			E(u_{1}) < E(u_{2}), \qquad
			0 < E(u_{2}) < c^{*}_{h},
		\end{equation*}
		for all  \(h \in W^{*}\backslash\{0\} \) with
		\(  \| h \|_{*} < \mu_{0} \).
	\end{theorem}
	
	There is also a simplified version of the abstract result above corresponding to  $0<\lambda<\lambda_{1}$.
	In this regime the  functional exhibits a natural mountain--pass geometry
	without the need for linking over a nontrivial index set, and the construction can be
	carried out by considering the case \(C_{\delta}=\emptyset\), yielding the
	following elementary variant of the abstract theorem.
	\begin{theorem}\label{thm:2.1 small lambda} 
		Let \(  \lambda \in(0, \lambda_{1} )\).
		Assume there exist  an element \( w \in \mathcal{M} \) and \( R > 0 \) such
		that
		\begin{align}
			E_{0}(Rw) \le 0 \quad\text{and}\quad \sup_{0 \le t \le R} E_{0}(tw) < c^{*}. \label{eq:2.91}
		\end{align}
		Then there exists \( \mu_{0} > 0 \) such that
		the functional $E$ defined in ~\eqref{eq def general E} possesses two distinct nontrivial critical points
		\( u_{1}, u_{2} \) satisfying
		\begin{equation*}
			E(u_{1}) < E(u_{2}), \qquad
			0 < E(u_{2}) < c^{*}_{h},
		\end{equation*}
		for all  \(h \in W^{*}\backslash\{0\} \) with
		\(  \| h \|_{*} < \mu_{0} \).
	\end{theorem}

	\section{On local Palais-Smale condition}\label{sec-PS}
	
	We recall that the variational functional $E$ given in \eqref{103} satisfies the Palais-Smale compactness condition at the level $c \in \R$, or the \PS{c} condition for short, if every sequence $\seq{w_n} \subset W$ satisfying $E(w_n) \to c$ and $E'(w_n) \to 0$, called a \PS{c} sequence for $E$, has a strongly convergent subsequence.
	
	As in \cite{DeMoraisFilho-Souto-99}, we show that, in the system setting, the threshold level $c$ for which the functional satisfies the $(\mathrm{PS})_c$ condition is determined by the $p^*$--homogeneous nonlinearity, rather than by the classical best constant of the Sobolev embedding into $L^{p^*}(\Omega)$. Nevertheless, there is a precise relationship between these two quantities, which we now describe.  Define
	\begin{eqnarray}\label{Sp}
		S:=\inf_{u\in W^{1,p}_0\backslash\set{0}}\dfrac{\int_{\Omega}|\nabla u|^pdx}{\seq{\int_{\Omega} |u|^{p^*}dx}^{p/p^*}}
	\end{eqnarray}
	and
	\begin{eqnarray}\label{best constant}
		S_H:=\inf_{w\in W\backslash\set{0}}\dfrac{\int_{\Omega}\seq{|\nabla u|^p+|\nabla v|^p}dx}{\seq{p^*\int_{\Omega} H(u,v)dx}^{p/p^*}}.
	\end{eqnarray}
	Let
	\begin{eqnarray}\label{bc1}
		M_H:=\max\set{\left(p^*H(s,t)\right)^{p/p^*}:|s|^p+|t|^p=1}
	\end{eqnarray}
	and $(s_0,t_0)\in\R^2$ such that
	\begin{eqnarray}\label{bc2}
		M_H^{-1}\left(p^*H(s_0,t_0)\right)^{p/p^*}=|s_0|^p+|t_0|^p=1.
	\end{eqnarray}
	The lemma below is closely related to \cite[Lemma~3]{DeMoraisFilho-Souto-99}. 
	Since in that work the $p^*$--homogeneous function is defined only for $s,t \ge 0$, some modifications are required in our setting. 
	For the sake of completeness, we provide a full proof here.
	\begin{lemma}\label{bc3}
		Assume \ref{H1}-\ref{H3}. 
		For $S,S_H$ defined in \eqref{Sp} and \eqref{best constant}, and  $M_H$ given in \eqref{bc1}   we have $$S_H=M_H^{-1}S.$$
	\end{lemma}
	\begin{proof}
		Let $\set{u_n}_{n\in\N}$ be a minimizing sequence for $S$, i.e.,
		\[\dfrac{\om|\nabla u_n|^p}{\seq{\om|u_n|^{p^*}}^{p/p^*}}\to S.
		\]
		We may assume $u_n\geq0$.   Consider the sequence $w_n=(s_0u_n,t_0u_n)$, where $(s_0,t_0)\in\R^2$ is as in $\eqref{bc2}$. Then
		\begin{eqnarray*}
			S_H\leq\dfrac{\om\seq{|\nabla s_0u_n|^p+|\nabla t_0 u_n|^p}}{\seq{p^*\om H(s_0u_n,t_0u_n)}^{p/p^*}}=M_H^{-1}\dfrac{\om|\nabla u_n|^p}{\seq{\om|u_n|^{p^*}}^{p/p^*}}.%\leq M_H^{-1}S,
		\end{eqnarray*}
		Letting $n\to\infty$ we obtain $S_H\leq M^{-1}_H S$. Now let us prove the other inequality; let $\set{w_n}_{n\in\N}$, $w_n=(u_n,v_n)$, be a minimizing sequence for $S_H$. Since $H$ satisfies \ref{H1} and \ref{H3}, we can apply the ``H\"older type inequality" from De Morais Filho and Souto \cite[Proposition 4]{DeMoraisFilho-Souto-99} to obtain
		\[
		p^*\om H(|u_n|,|v_n|)dx\leq H(\|u_n\|_{p^*},\|v_n\|_{p^*})\quad\forall n\in\mathbb{N}.
		\]
		Then, using \ref{H2}  we see that
		\begin{eqnarray*}
			\dfrac{\om\seq{|\nabla u_n|^p+|\nabla v_n|^p}dx}{\seq{p^*\om H(u_n,v_n)dx}^{p/p^*}}\geq S\dfrac{\|u_n\|_{p^*}^p+\|v_n\|_{p^*}^p}{(p^*H(\|u_n\|_{p^*},\|v_n\|_{p^*}))^{p/p^*}}\geq M_H^{-1}S,
		\end{eqnarray*}
		and after passing to the limit, we have $S_H\geq M_H^{-1}S$. This concludes the proof.
	\end{proof}
	
	\begin{proposition}\label{bc4}
		Assume \ref{F1}-\ref{F3}, \ref{G subcri}-\ref{AR G} and \ref{H1}-\ref{H3}. 
		There exists a constant $c(\gamma,\eta)>0$, with $c(\gamma,\eta)\to0$ as  $\gamma,\eta\to0$ in $L^{(p^\ast)'}(\Omega)$, such that $E$ satisfies the $(PS)_c$ condition for all $c<S_H^{N/p}/N-c(\gamma,\eta)$. In particular, $E_0$ satisfies  $(PS)_c$  for all $c<S_H^{N/p}/N$.
	\end{proposition}
	\begin{proof}
		Let $w_n=(u_n,v_n)$ be a $(PS)_c$ sequence for $E$ in $W$, for  $c<S_H^{N/p}/N$. This implies that $E(w_n)=c+o(1)$ and $E^\prime(w_n)\xi=o({\norm{\xi}})$ for all $\xi\in W$. Then
		\begin{align}\label{6-1}
			\frac{1}{p}\int_{\Omega}\seq{|\nabla u_n|^p+|\nabla v_n|^p}dx-&\lambda\int_{\Omega} F(u_n,v_n)dx-\int_{\Omega} H(u_n,v_n)dx-\int_{\Omega} G(u_n,v_n)dx\nonumber\\
			&-\int_{\Omega}\seq{\gamma u_n+\eta v_n}=c+o(1)
		\end{align}
		and
		\begin{multline} \label{207}
			\int_\Omega (|\nabla u_n|^{p-2}\, \nabla u_n \cdot \nabla \varphi + |\nabla v_n|^{p-2}\, \nabla v_n \cdot \nabla \psi)\, dx - \lambda \int_\Omega (F_u(u_n,v_n)\, \varphi + F_v(u_n,v_n)\, \psi)\, dx\\
			- \int_\Omega (H_u(u_n,v_n)\, \varphi + H_v(u_n,v_n)\, \psi)\, dx -\int_\Omega (G_u(u_n,v_n)\, \varphi + G_v(u_n,v_n)\, \psi)\, dx\\
			-\int_{\Omega}\seq{\gamma \varphi+\eta \psi}dx
			= \o(\norm{\xi})
		\end{multline}
		for all $\xi = (\varphi,\psi) \in W$. First we show that $\seq{w_n}$ is bounded in $W$. Taking $\xi = w_n$ in \eqref{207} and using \eqref{201} for  $F$ and $H$ give
		\begin{align}\label{6-2}
			\int_{\Omega}\seq{|\nabla u_n|^p+|\nabla v_n|^p}-&\lambda p\int_{\Omega} F(u_n,v_n)-p^*\int_{\Omega} H(u_n,v_n)dx-\int_{\Omega} \nabla G(u_n,v_n)(u_n,v_n)dx\nonumber\\
			&-\int_{\Omega}\seq{\gamma u_n+\eta v_n}dx=o(\|w_n\|).
		\end{align}
		Let $r\in(p,p^*)$ as in \ref{AR G}. After multiplying \eqref{6-2} by $1/r$ and subtracting it from \eqref{6-1} we have
		\begin{align*}
			& \seq{\frac{1}{p}-\dfrac{1}{r}}\int_{\Omega}\seq{|\nabla u_n|^p+|\nabla v_n|^p}dx+\lambda\seq{\dfrac{p}{r}-1}\int_{\Omega} F(u_n,v_n)dx\nonumber\\
			&\,\,+\seq{\dfrac{p^*}{r}-1}\int_{\Omega} H(u_n,v_n)dx
			%\frac{1}{r}\int_{\Omega} \left[\nabla G(u_n,v_n)(u_n,v_n)-r G(u_n,v_n)\right]dx+
			+ \seq{\dfrac{1}{r}-1}\int_{\Omega}\seq{\gamma u_n+\eta v_n}dx\le c+o(1)\|w_n\|.
		\end{align*}
		So, due to \ref{H2} and since $F$ satisfies \eqref{202}  we obtain
		\begin{eqnarray}
			\label{6-3}
			\seq{\dfrac{1}{p}-\dfrac{1}{r}}\|w_n\|^p&\!\leq\!&\lambda\seq{1-\dfrac{p}{r}}C_F\int_{\Omega}\seq{|u_n|^p+|v_n|^p}dx
			-\seq{\dfrac{p^*}{r}-1}c_H\int_{\Omega}\seq{|u_n|^{p^*}+|v_n|^{p^*}}dx\nonumber\\
			&&+\seq{1-\dfrac{1}{r}}
			\seq{\|\gamma\|_{p^*\prime}\|u_n\|_{p^*}+\|\eta\|_{p^*\prime}\|v_n\|_{p^*}}
			+c+o(1)\|w_n\|.     
		\end{eqnarray}
		For simplicity let us define
		\[
		\alpha_1:=\dfrac{1}{p}-\dfrac{1}{r},\quad\alpha_2:=\lambda\seq{1-\dfrac{p}{r}}C_F,\quad\alpha_3:=\seq{\dfrac{p^*}{r}-1}c_H,\quad\alpha_4:=1-\dfrac{1}{r}
		\]
		so that \eqref{6-3} can be written as
		\begin{eqnarray}\label{boundness}
			\alpha_1\|w_n\|^p&\leq& \alpha_2\seq{\|u_n\|_p^p+\|v_n\|_p^p}-\frac{\alpha_3}{2}\seq{\|u_n\|_{p^*}^{p^*}+\|v_n\|_{p^*}^{p^*}}
			+\seq{\alpha_4\|\gamma\|_{p^*\prime}\|u_n\|_{p^*}-\frac{\alpha_3}{2}\|u_n\|_{p^*}^{p^*}}\nonumber\\
			&&+\seq{\alpha_4\|\eta\|_{p^*\prime}\|v_n\|_{p^*}-\frac{\alpha_3}{2}\|v_n\|_{p^*}^{p^*}}
			+c+o(1)\|w_n\|.
		\end{eqnarray}
		Notice that there exists a constant $K=K(\lambda,p,N,r, C_F,c_H)>0$ such that
		\begin{eqnarray*}
			\alpha_2\seq{\|u\|_p^p+\|v\|_p^p}-\frac{\alpha_3}{2}\seq{\|u\|_{p^*}^{p^*}+\|v\|_{p^*}^{p^*}}\leq K\quad\forall (u,v)\in W.
		\end{eqnarray*}
		and this shows that the sequence $\|w_n\|$ is bounded in $W$. 
		Passing to a subsequence, we may now assume that $\seq{w_n}$ converges to some $w = (u,v)$ weakly in $W$, strongly in $L^q(\Omega) \times L^s(\Omega)$ for all $q, s\in [1,p^\ast)$, and a.e.\! in $\Omega$. Set $\widetilde{u}_n = u_n - u$, $\widetilde{v}_n = v_n - v$, and $\widetilde{w}_n = (\widetilde{u}_n,\widetilde{v}_n) = w_n - w$. 
		%We will show that if $\norm{\widetilde{w}_n} \ge \eps_0 > 0$, then \eqref{204} does not hold, and conclude that $w_n \to w$ in $W$ for a further subsequence.
		Since \eqref{202} holds for $F$ and $w_n\to w$ in $L^p(\Omega) \times L^p(\Omega)$, and $w_n\rightharpoonup w$ in $W$, we have
		\[
		\int_\Omega F(u_n,v_n)\, dx \to \int_\Omega F(u,v)\, dx,\quad\text{and}\quad \int_{\Omega}\seq{\gamma u_n+\eta v_n}dx\to\int_{\Omega}\seq{\gamma u+\eta v}dx.
		\]
		Recalling that $G$ satisfies \ref{G subcri} we also get 
		\[
		\int_\Omega G(u_n,v_n)\, dx \to \int_\Omega G(u,v)\, dx\quad\text{and}\quad\int_\Omega (u_n,v_n)\cdot \nabla G(u_n,v_n)\, dx \to \int_\Omega(u,v)\cdot\nabla G(u,v)\, dx.
		\]
		So \eqref{6-1} and \eqref{6-2} reduce to
		\begin{multline} \label{209}
			\frac{1}{p} \int_\Omega (|\nabla u_n|^p + |\nabla v_n|^p)\, dx - \lambda \int_\Omega F(u,v)\, dx - \int_\Omega H(u_n,v_n)\, dx \\
			-\int_\Omega G(u,v)\, dx -\int_{\Omega}\seq{\gamma u+\eta v}dx= c + \o(1)
		\end{multline}
		and
		\begin{multline} \label{210}
			\int_\Omega (|\nabla u_n|^p + |\nabla v_n|^p)\, dx - \lambda p \int_\Omega F(u,v)\, dx - p^\ast \int_\Omega H(u_n,v_n)\, dx \\
			-\int_\Omega (u,v)\cdot\nabla G(u,v)\, dx -\int_{\Omega}\seq{\gamma u+\eta v}dx= \o(1),
		\end{multline}
		respectively. Multiplying \eqref{210} by $1/p^*$ and subtracting it from \eqref{209} gives
		\begin{eqnarray}\label{6-5}
			\dfrac{1}{N}\int_{\Omega}\seq{|\nabla u_n|^p+|\nabla v_n|^p}dx-\lambda\dfrac{p}{N}\int_{\Omega} F(u,v)dx+\seq{\dfrac{1}{p^*}-1}\int_{\Omega}\seq{\gamma u+\eta v}dx\nonumber\\
			-\int_{\Omega} G(u,v)dx+\dfrac{1}{p^*}\int_{\Omega}(u,v)\cdot \nabla G(u,v)dx=c+o(1).
		\end{eqnarray}
		Taking $\psi = 0$ in \eqref{207} and applying Boccardo and Murat \cite[Theorem 2.1 and Remark 2.1]{boccardo-murat} %{MR1183665} 
		shows that $\nabla u_n \to \nabla u$ a.e.\! in $\Omega$ for a renamed subsequence. Similarly, $\nabla v_n \to \nabla v$ a.e.\! in $\Omega$ for a further subsequence. So taking $\xi = w$ in \eqref{207}, passing to the limit, and using \eqref{201} for $F$ and $H$ gives
		\begin{multline}\label{6-6}
			\int_\Omega (|\nabla u|^p + |\nabla v|^p)\, dx - \lambda p \int_\Omega F(u,v)\, dx - p^\ast \int_\Omega H(u,v)\, dx\\
			-\int_\Omega (u,v)\cdot\nabla G(u,v)\, dx- \int_\Omega (\gamma u+\eta v)\, dx= 0.
		\end{multline}
		We have
		\[
		\int_\Omega (|\nabla u_n|^p + |\nabla v_n|^p - |\nabla u|^p - |\nabla v|^p)\, dx = \int_\Omega (|\nabla \widetilde{u}_n|^p + |\nabla \widetilde{v}_n|^p)\, dx + \o(1)
		\]
		and
		\[
		\int_\Omega (H(u_n,v_n) - H(u,v))\, dx = \int_\Omega H(\widetilde{u}_n,\widetilde{v}_n)\, dx + \o(1)
		\]
		by the Br{\'e}zis-Lieb lemma (see \cite[Theorem 1]{Brezis-Lieb-1983} and \cite[Lemma 8]{DeMoraisFilho-Souto-99}).
		Now, multiplying \eqref{6-6} by $1/N$ and subtracting it from \eqref{6-5} one gets
		\begin{eqnarray*}
			\dfrac{1}{N}\int_{\Omega}\seq{|\nabla \widetilde{u}_n|^p+|\nabla\widetilde{v}_n|^p}dx+\frac{p^*}{N}\int_{\Omega} H(u,v) dx+\seq{\dfrac{1}{p}-1}\int_{\Omega}(\gamma u+\eta v) dx\nonumber\\
			-\int_{\Omega} G(u,v)dx+\dfrac{1}{p}\int_{\Omega}(u,v)\cdot \nabla G(u,v)dx=c+o(1).
		\end{eqnarray*}
		Since $G$ satisfies \ref{AR G} and $H\ge0$, it follows that
		\begin{eqnarray}\label{6-7}
			\dfrac{1}{N}\int_{\Omega}\seq{|\nabla \widetilde{u}_n|^p+|\nabla\widetilde{v}_n|^p}dx\le\seq{1-\dfrac{1}{p}}\int_{\Omega}(\gamma u+\eta v) dx+c+o(1).
		\end{eqnarray}
		Now notice that
		\begin{eqnarray*}
			\int_\Omega|\gamma u+\eta v|dx\leq\|\gamma\|_{(p^{*})'}S^{-1/p}\seq{\int_\Omega|\nabla u|^p}^{1/p}+\|\eta\|_{(p^{*})'}S^{-1/p}\seq{\int_\Omega |\nabla v|^p}^{1/p}\\
			\leq 2S^{-1/p}\seq{\|\gamma\|_{(p^{*})'}+\|\eta\|_{(p^{*})'}}\|w\|.
		\end{eqnarray*}
		By \eqref{boundness}, there exists a constant $K_2>0$ independent of $c<S_H^{N/p}/N$ and of the $(PS)_c$ sequence $\seq{w_n}$, such that $\|w_n\|\leq K_2$, and since $w_n\rightharpoonup w$, we have
		\(
		\|w\|\leq K_2.
		\)
		Consequently,
		\[
		\int_\Omega|\gamma u+\eta v|\leq2S^{-1/p}K_2\seq{\|\gamma\|_{(p^{*})'}+\|\eta\|_{(p^{*})'}}.
		\]
		This together with \eqref{6-7} gives us
		\begin{eqnarray}\label{6-8}
			\dfrac{1}{N}\int_\Omega\seq{|\nabla \widetilde{u}_n|^p+|\nabla\widetilde{v}_n|^p}-2S^{-1/p}K_2\seq{\|\gamma\|_{(p^{*})'}+\|\eta\|_{(p^{*})'}}\leq c+o(1).
		\end{eqnarray}
		On the other hand, combining \eqref{6-2} and \eqref{6-6} gives us 
		\begin{eqnarray}\label{6-9}
			\int_\Omega\seq{|\nabla \widetilde{u}_n|^p+|\nabla\widetilde{v}_n|^p}=p^*\int_\Omega H(\widetilde{u}_n,\widetilde{v}_n)+o(1).
		\end{eqnarray}
		Now, suppose $\|\widetilde{w}_n\|\not\to 0$, then we can assume that there exists $\varepsilon>0$ such that 
		$\|\widetilde{w}_n\|^p\geq\varepsilon$ for large $n$. So, \eqref{6-9} together with \eqref{best constant} implies that
		\[
		\dfrac{1}{N}S_H^{N/p}\leq\dfrac{1}{N}\int_\Omega\seq{|\nabla \widetilde{u}_n|^p+|\nabla\widetilde{v}_n|^p}+o(1).
		\]
		Consequently, using  \eqref{6-8} we reach
		\[
		\dfrac{1}{N}S_H^{N/p}-2S^{-1/p}K_2\seq{\|\gamma\|_{(p^{*})'}+\|\eta\|_{(p^{*})'}}\leq c.
		\]
		Therefore, defining $c(\gamma,\eta):=2S^{-1/p}K_2\seq{\|\gamma\|_{(p^{*})'}+\|\eta\|_{(p^{*})'}}$ we conclude this proof.
	\end{proof}

	\section{Linking geometry and proof of main theorems}\label{sec-geo}
	
	In  Proposition \ref{bc4} 
	we have proved that the functional $E$, defined in \eqref{103}, satisfies the $(PS)_c$ condition for $c<c^*+c(\gamma,\eta)$, with $c^*=S_H^{N/p}/N$ and $c(\gamma,\eta)\to0$ as $\gamma,\eta\to0$ in $L^{(p^*)'}(\Omega)$. In particular, for $\gamma=\eta=0$ the functional $E_0$
	given in $\eqref{E0 def}$ satisfies the $(PS)_c$ condition for $c<c^*$.  In this section we will show that the functional $E_0$ satisfies the additional conditions of Theorems \ref{lk th} and $\ref{thm:2.1}$, when $\lambda\ge\lambda_1$.

	Let
	\begin{eqnarray*}
		u_\varepsilon=\dfrac{c_{N,p}\varepsilon^{(N-p)/p^2}}{(\varepsilon+|x|^{p/(p-1)})^{(N-p)/p}},\quad\varepsilon>0
	\end{eqnarray*}
	where $c_{N,p}>0$ is a constant such that
	\begin{eqnarray*}
		\int_{\R^N}|\nabla u_\varepsilon|^pdx=\int_{\R^N}u_{\varepsilon}^{p^*}dx=S^{N/p}.
	\end{eqnarray*}
	Suppose $0\in\Omega$. Let $\delta_0=\text{dist}(0,\partial\Omega)$, let $\zeta:[0,\infty)\to[0,1]$ be a smooth function such that $\zeta(t)=1$ for $t\leq1/4$, and $\zeta(t)=0$ for $t\geq1/2$. Define
	\begin{eqnarray*}
		u_{\varepsilon,\delta}(x)=\zeta(|x|/\delta)u_\varepsilon(x)\quad\text{and}\quad
		z_{\varepsilon,\delta}(x)=\dfrac{u_{\varepsilon,\delta}(x)}{\seq{\int_{\R^N} u_{\varepsilon,\delta}^{p^*}dx}^{1/p^*}},\quad0<\delta\leq\delta_0/2,
	\end{eqnarray*}
	so that
	\begin{eqnarray}\label{estimate1}
		\om z_{\varepsilon,\delta}^{p^*}dx=1,\quad\om|\nabla z_{\varepsilon,\delta}|^pdx\leq S+c_1\varepsilon^{\frac{N-p}{p}}\delta^{-\frac{N-p}{p-1}},
	\end{eqnarray}
	for $S$ as in \eqref{Sp}, and
	\begin{eqnarray}\label{estimate2}
		\om z_{\varepsilon,\delta}^pdx\geq\begin{cases}
			c_2\varepsilon^{p-1}&\quad\text{if}\quad N>p^2\\
			c_2\varepsilon^{p-1}|\log\seq{\varepsilon\delta^{-p/(p-1)}}|&\quad\text{if}\quad N=p^2\\
			c_2\varepsilon^{\frac{N-p}{p}}\delta^{N-\frac{p(N-p)}{p-1}}&\quad\text{if}\quad N<p^2
		\end{cases}
	\end{eqnarray}
	for some constants $c_1,c_2>0$  (see \cite{Perera-Zou2018}). We also have the following estimate, by \cite[Lemma A5]{Garcia-Azorero-Peral1994} (see also \cite[Lemma 3.1]{doO-Gloss-Ribeiro2020} and \cite[Lemma 11.1]{MR1695021}):
	\begin{eqnarray}\label{estimate3}
		\int_{\R^N} z_{\varepsilon,\delta}^qdx\geq \widetilde{C}\varepsilon^{\seq{\frac{p-1}{p}}\seq{N-\frac{q(N-p)}{p}}}&\quad\text{if}\quad q>p^*\seq{1-\frac{1}{p}}.
	\end{eqnarray}
	Recall that $\pi(w)= w/\|w\|$ for $w\in W\backslash\{0\}$ and let
	\begin{eqnarray}\label{a}
		e_{\varepsilon,\delta}=\pi(s_0z_{\varepsilon,\delta},t_0z_{\varepsilon,\delta})=(s_0\w{z}_{\varepsilon,\delta},t_0\w{z}_{\varepsilon,\delta})
	\end{eqnarray}
	where $(s_0,t_0)\in\R^2$ satisfies \eqref{bc2} and
	\begin{eqnarray}\label{b}
		\w{z}_{\varepsilon,\delta}=\dfrac{z_{\varepsilon,\delta}}{\|\nabla z_{\varepsilon,\delta}\|_{p}}.
	\end{eqnarray}
	Assuming $\lambda_k\le\lambda<\lambda_{k+1}$, due to Proposition \ref{Proposition 2} we know that there exists 
	$\mathcal{C}\subset{\Psi}^{\lambda_k}$ a compact symmetric subset of index $k$ that is bounded in $(L^\infty(\Omega)\times L^\infty(\Omega))\cap (C_{\text{loc}}^1(\Omega)\times C_{\text{loc}}^1(\Omega))$. Let $\ell:[0,\infty)\to[0,1]$ be a smooth function such that 
	$$
	\ell(t)=0\quad\text{for}\,\,t\leq 3/4\quad\text{and}\quad \ell(t)=1\quad\text{for}\,\,t\geq1.
	$$ 
	For $0<\delta\leq\delta_0/2$ we denote
	\begin{equation*}
		\overline{w}_\delta=\seq{\ell(|x|/\delta)u(x),\ell(|x|/\delta)v(x)},\quad w=(u,v)\in \mathcal{C},
	\end{equation*}
	and define
	\begin{equation}\label{C_delta def}
		C_\delta=\set{\pi(\overline{w}_\delta): w\in \mathcal{C}}.
	\end{equation}
	
	Using similar arguments as in Bisci et al. \cite[Lemma 4.1]{Bisci&all-2025}, due to the $p$-homogeneity and continuity of $F$ and the regularity of functions in $\mathcal{C}$ given in Proposition \ref{Proposition 2}, we can verify that $C_\delta$ has the following properties.
	\begin{lemma}\label{lemma1}
		Assume \ref{F1}-\ref{F3}.
		There exists $\delta_1\in(0,\delta_0/2]$  such that for all $0<\delta\leq\delta_1$,
		\begin{enumroman}
			\item\label{lemma1-1} $C_\delta$ is a compact symmetric subset of ${\Psi}^{\lambda_k+c_3\delta^{N-p}}$ for some constant $c_3>0$;
			\item\label{lemma1-2} $\lambda_k+c_3\delta^{N-p}<\lambda_{k+1}$ and $C_\delta\cap{\Psi}_{\lambda_{k+1}}=\emptyset$;
			\item\label{lemma1-3} $i(C_\delta)=k$;
			\item\label{lemma1-4} $e_{\varepsilon,\delta}\in\M\backslash{C_\delta}$, for all $\varepsilon>0$;
		\end{enumroman}
	\end{lemma}
	These conditions on $C_\delta$ allow us to get a uniform estimate for $E_0$, when $\delta\in(0,\delta_1]$.
	\begin{lemma}\label{lemma1-5}
		Assume \ref{F1}-\ref{F3}, \ref{G subcri}-\ref{AR G} and \ref{H1}-\ref{H2}. Then,  there is $R>0$ such that 
		$$E_0(R\widetilde{w})\leq0,\quad\forall \widetilde{w}\in A_\delta:=\left\{\pi((1-t)w+te_{\varepsilon,\delta}):\, w\in C_\delta,\, t\in[0,1]\right\}$$
		for 
		all $0<\delta\leq\delta_1$ and  $\varepsilon>0$ such that $\varepsilon^\frac{N-p}{p}\delta^{-\frac{N-p}{p-1}}\le 1$.
	\end{lemma}
	\begin{proof}
		Let $\widetilde{w} = \pi((1 - t)w + t e_{\varepsilon,\delta})\in A_\delta$, with $\widetilde{w} = (\widetilde{u}, \widetilde{v})$. Since $Q=H+G\ge H$ and $H$ satisfies \ref{H2}, we have
		\begin{align}\label{item v geom}
			E_0(R\widetilde{w})&=\dfrac{R^p}{p}\om\seq{|\nabla \w{u}|^p+|\nabla \w{v}|^p}dx-\lambda R^p\om F(\widetilde{u}, \widetilde{v})dx-\om Q(R\widetilde{u}, R\widetilde{v})dx\nonumber\\
			&\leq \frac{R^p}{p}  
			- c_H{R^{p^{\ast}}} \int_\Omega \left(|\widetilde{u}|^{p^{\ast}}+|\widetilde{v}|^{p^{\ast}}\right)dx.
		\end{align}
		Since $e_{\varepsilon,\delta}$ and $w\in\mathcal{M}$  and have disjoint support for all $w\in C_\delta$, we can see that
		\begin{align}\label{tilde v inf}
			\int_\Omega\left(|\widetilde{u}|^{p^{\ast}}+|\widetilde{v}|^{p^{\ast}}\right)dx \geq c_0
			\min\left\{\int_\Omega( |u|^{p^{\ast}}+ |v|^{p^{\ast}})dx;\int_\Omega |\widetilde{z}_{\varepsilon,\delta}|^{p^{\ast}}dx\right\}
		\end{align}
		for
		\[
		c_0=(|s_0|^{p^*}+|t_0|^{p^*})\min_{t\in[0,1]}\frac{(1-t)^{p^{\ast}}+t^{p^{\ast}}}{\left[(1-t)^p+ t^p\right]^{{p^{\ast}}/p}}>0.
		\]
		Due to  \eqref{estimate1} and \eqref{b} we see that, for $\varepsilon^\frac{N-p}{p}\delta^{-\frac{N-p}{p-1}}\le 1$, there exists $\widetilde{c}_1=\widetilde{c}_1(N,p,\Omega)>0$ such that
		\begin{equation}\label{z control}
			\int_\Omega |\widetilde{z}_{\varepsilon,\delta}|^{p^{\ast}}\,dx\geq \widetilde{c}_1.
		\end{equation}
		On the other hand, by Lemma \ref{lemma1} \ref{lemma1-2} we have $C_\delta\subset \mathcal{M}\backslash{\Psi}_{\lambda_{k+1}}$, which implies
		\[
		p\int_\Omega F(u,v)\,dx=\frac{1}{\Psi(u,v)}> \frac{1}{\lambda_{k+1}},\quad\forall (u,v)\in C_\delta.
		\]
		Recalling that $F$ satisfies \eqref{202}, we see that
		\[
		p\int_\Omega F(u,v)\,dx\leq C\left(\int_\Omega |u|^{p^{\ast}}\,dx\right)^\frac{p}{p^{\ast}}+C\left(\int_\Omega |v|^{p^{\ast}}\,dx\right)^\frac{p}{p^{\ast}},\quad\forall (u,v)\in C_\delta
		\]
		with $C=C(N,p,\Omega,C_F)>0$.
		Then, there is a positive constant $\widetilde{c}_2$, which does not depend on $\delta$ nor $\varepsilon$, such that
		\begin{equation}\label{v inf Ad}
			\int_\Omega \left(|u|^{p^{\ast}}+|v|^{p^{\ast}}\right) dx\ge \widetilde{c}_2,\quad\forall (u,v)\in C_\delta. 
		\end{equation}
		Thus, \eqref{tilde v inf}, \eqref{z control} and \eqref{v inf Ad} ensure that 
		\begin{equation*}%\label{tilde v bounded bellow}
			\int_\Omega\left(|\tilde u|^{p^{\ast}}+|\tilde v|^{p^{\ast}}\right)dx\geq \widetilde{c}_3
		\end{equation*} 
		for a constant $\widetilde{c}_3>0$ independent of $R$ and $\widetilde{w}$.
		Therefore, from \eqref{item v geom} we obtain
		\begin{equation*}
			E_0(R\widetilde w) \le \frac{R^p}{p} - \widetilde{c}_3{R^{p^{\ast}}}<0
		\end{equation*}
		for $R>0$ large enough. This completes the proof.
	\end{proof}
	
	%%%%%%%%%%%%%%%%%%%%%%%%%%%%%%%%%%%%%%%%%%%%%%%%%%%%%%%%%%%%%%%%%%%%%%%%%%%%%%%%%%%%%%%%%%%%%
	
	In order to verify the geometry of Theorem \ref{lk th}, let us consider
	\[
	A_0 = C_\delta, \quad B_0 = \Psi_{\lambda_{k+1}}, \quad e = e_{\varepsilon,\delta},
	\]
	as in Lemma \ref{lemma1}, with $\delta\in(0,\delta_1)$ to be chosen later. By Lemma \ref{lemma1}, $A_0$ is a compact symmetric subset of $\Psi^{\lambda_k+c_3\delta^{N-p}}$, $A_0\subset \mathcal{M} \setminus \Psi_{\lambda_{k+1}}$, $i(A_0)=k$ and $e \in \mathcal{M} \setminus A_0$. Due to Theorem \ref{Theorem 4} \ref{Theorem 4.iii} we have
	\(
	i(\mathcal{M} \setminus B_0) = i(\mathcal{M} \setminus \Psi_{\lambda_{k+1}}) = k.
	\)
	Thus,
	\begin{equation}\label{same index}
		i(A_0)=i(\mathcal{M} \setminus B_0).
	\end{equation}
	For this $A_0$, $R > 0$ as given in Lemma \ref{lemma1-5} and $\rho\in(0,R)$ we define
	\[
	A = \{ t w : w \in A_0, \, t \in [0,R] \} \cup \{ R \pi((1 - t)w + t e) : w \in A_0,\, t \in [0,1] \},
	\]
	\begin{align}\label{X def}
		B = \{ \rho w : w \in B_0 \},
		\quad\text{and}\quad
		X = \{ t w : w \in A, \, \|w\| = R,\, t \in [0,1] \}.
	\end{align}
	Since \ref{G subcri} holds, there exist $C>0$ and $q_1,q_2\in(p,p^*)$ such that
	\[
	0\leq G(s,t)\leq C(|s|^{q_1}+|s|^{q_2}+|t|^{q_1}+|t|^{q_2})\quad\forall s,t\in\R.
	\]
	Then, recalling that $p \int_\Omega F(w)\,dx=1/\Psi(w)$ for $w\in\mathcal{M}$, for $w=(u,v) \in B_0$ it holds
	\begin{align}\label{E(tw) M}
		E_0(tw) &= \frac{t^p}{p} \left( 1 - \frac{\lambda}{\Psi(w)} \right) - \int_\Omega G(tu,tv)\,dx-t^{p^*}\int_\Omega H(u,v)\,dx\nonumber\\
		&\geq\frac{t^p}{p} \left( 1 - \frac{\lambda}{\lambda_{k+1}} \right) - c_Ht^{p^{\ast}} \int_\Omega (|u|^{p^{\ast}}+|v|^{p^{\ast}})dx\nonumber\\
		&\quad- Ct^{q_1}\int_\Omega (|u|^{q_1}+|v|^{q_1})dx-Ct^{q_2}\int_\Omega (|u|^{q_2}+|v|^{q_2})dx.
	\end{align}
	Due to the boundedness of $\mathcal{M}$, it is easy to see that  there is $\rho> 0$ small enough such that
	\[
	\inf_{w \in B} E_0(w)= \inf_{w \in B_0} E_0(\rho w)> 0.
	\]
	Let $w \in A_0$. Since $ A_0\subset\Psi^{\lambda_k+c_3\delta^{N-p}}$ and $\lambda_k \leq \lambda$, we have
	\begin{align*}%\label{sup in tA0}
		E_0(tw) &= \frac{t^p}{p} \left( 1 - \frac{\lambda}{\Psi(w)} \right) - \int_\Omega G(tu,tv)\,dx-t^{p^*}\int_\Omega H(u,v)\,dx%\leq \frac{t^p}{p} \left( 1 - \frac{\lambda}{\lambda_k + c_3 \delta^{N-p}} \right)  
		\leq \frac{R^p}{p} \cdot \frac{c_3 \delta^{N-p}}{\lambda_k },
	\end{align*}
	for all $t \in [0, R]$.
	Therefore, if $\delta> 0$ is small enough, then
	\begin{equation}\label{supA0<infB}
		\sup_{w \in A_0,\,t\in[0,R]} E_0(tw) < \inf_{w \in B} E_0(w).
	\end{equation}
	Joining this information with Lemma \ref{lemma1-5}  we conclude that
	\begin{equation}\label{supA<infB}
		\sup_{w \in A} E_0(w) < \inf_{w \in B} E_0(w).
	\end{equation}
	It remains to estimate $E$ on $X$. Recalling that $w$ and $e=e_{\varepsilon,\delta}$ have disjoint support for all $w\in A_0$,  we have
	\begin{align}\label{sup on X}
		\sup_{w\in X} E_0(w) = \sup_{w \in A_0,\, s,t \geq 0} E_0(sw + t e_{\varepsilon,\delta}) 
		= \sup_{w \in A_0,\,s \geq 0} E_0(sw) + \sup_{t \geq 0} E_0(t e_{\varepsilon,\delta}). 
	\end{align}
	
	In the next lemmas we will estimate these two suprema on the right side of \eqref{sup on X}.

	%%%%%%%%%%%%%%%%%%%%%%%%%%%%%%%%%%%%%%%%%%%%%%%%%%%%%%%%%%%%%%%%%%%%%%%%%%%%%
	
	\begin{lemma}\label{lemma2-1}
		Assume \ref{F1}-\ref{F3}, \ref{G subcri}, \ref{AR G} and \ref{H1}-\ref{H3}. Then we have
		\begin{eqnarray*}
			\sup_{w\in C_\delta,s\geq0}E_0(sw)\leq\begin{cases}
				0&\quad\text{if}\quad\lambda_k+c_3\delta^{N-p}\leq\lambda<\lambda_{k+1},\\
				c_6\delta^{N(N-p)/p}&\quad\text{if}\quad\lambda=\lambda_k,
			\end{cases}
		\end{eqnarray*}
		where $c_3$ is as in Lemma \ref{lemma1} and $c_6$ is a positive constant.
	\end{lemma}
	\begin{proof}
		For $w=(u,v)\in C_\delta$, and $s\geq0$, since $G\ge0$ we have
		\begin{eqnarray}\label{g2}
			E_0(sw)&=&\dfrac{s^p}{p}\seq{\om\seq{|\nabla u|^p+|\nabla v|^p}dx-p\lambda\om F(u,v)dx}-\om Q(su,sv)dx\nonumber\\
			&\leq&\dfrac{s^p}{p}\seq{1-\dfrac{\lambda}{\Psi(w)}}-s^{p^*}\om H(u,v)dx\nonumber\\
			&\leq&\dfrac{s^p}{p}\seq{1-\dfrac{\lambda}{\lambda_k+c_3\delta^{N-p}}}-s^{p^*}\om H(u,v)dx
		\end{eqnarray}
		once $C_\delta\subset\Psi^{\lambda_k+c_3\delta^{N-p}}$ by Lemma \ref{lemma1} \ref{lemma1-1}. So $E_0(sw)\leq0$ if $\lambda_k+c_3\delta^{N-p}\leq\lambda<\lambda_{k+1}$. Now if $\lambda_k=\lambda$, then
		\[
		1-\dfrac{\lambda}{\lambda_k+c_3\delta^{N-p}}=\dfrac{c_3\delta^{N-p}}{\lambda_k+c_3\delta^{N-p}}\leq c_4\delta^{N-p},
		\]
		where $c_4=c_3/\lambda_k>0$. For a  constant  $c_5>0$ independent of $\delta$ we also have 
		\[
		\om H(u,v)dx\geq c_5\quad\forall (u,v)\in C_\delta,
		\]
		since $H$ satisfies \ref{H2} and \eqref{v inf Ad} occurs. From this and $\eqref{g2}$ we get
		\begin{eqnarray}\label{g22}
			E_0(sw)\leq c_4\delta^{N-p}\dfrac{s^p}{p}-c_5s^{p^*}\leq c_6\delta^{N(N-p)/p}, \quad\forall w\in C_\delta, \forall s\geq0.
		\end{eqnarray}
		This  completes the proof.
	\end{proof}
	
	Now, we are going to estimate the second supremum on the right side of \eqref{sup on X}. Considering $\delta_1>0$ smaller than the one given  in Lemma \ref{lemma1}, and satisfying \eqref{supA0<infB}, if necessary, we may assume  that $c_6\delta^{N(N-p)/p}<\frac{1}{2N}S_H^{N/p}$ for all $\delta\in(0,\delta_1]$, where $S_H$ is given in \eqref{best constant}.
	We now have two possibilities: 
	$$
	\sup_{t\geq0}E_0(t\e)<\frac{1}{2N}S_H^{N/p}\quad\text{or}\quad \sup_{t\geq0}E_0(t\e)\geq\frac{1}{2N}S_H^{N/p}.
	$$
	If the former happens, this together with \eqref{g22} and with the choice of $\delta$ give both inequalities \eqref{lk th2} and \eqref{eq:2.9} and this shows that $\sup_{w\in X}E_0(w)<\frac{1}{N}S_H^{N/p}$ and Theorem \ref{lk th} can be employed to prove Theorem \ref{Theorem 1}. Theorem \ref{thm:2.1} can also be employed in this case to prove Theorem \ref{Theorem 5}, where $C_\delta$ is given in \eqref{C_delta def} and $R>0$ is given by Lemma \ref{lemma1-5}.
	
	From now on we suppose  that 
	\begin{equation}\label{sup large}
		\sup_{t\geq0}E_0(t\e)\geq\frac{1}{2N}S_H^{N/p}.
	\end{equation}
	By \ref{G subcri}-\ref{AR G}, \eqref{estimate1} and \eqref{estimate2}, one can easily deduce that $\lim_{t\to0^+}E_0(t\e)=0$ uniformly on  $0<\delta\leq\delta_1$ and  $\varepsilon>0$ such that $\varepsilon^\frac{N-p}{p}\delta^{-\frac{N-p}{p-1}}\le 1$. Therefore there exists $t^*>0$ such that 
	\begin{equation*}%\label{small t}
		E_0(t\e)<\frac{1}{2N}S_H^{N/p}\quad\forall t\in[0,t^*],
	\end{equation*}
	for all $0<\delta\leq\delta_1$ and  $\varepsilon>0$ small such that $\varepsilon^\frac{N-p}{p}\delta^{-\frac{N-p}{p-1}}\le 1$. On the other hand, for each  $\varepsilon,\delta>0$ it holds $\lim_{t\to\infty}E_0(t\e)=-\infty$.
	Consequently, given $\varepsilon,\delta>0$ small, there exists $\widetilde{t}=\widetilde{t}(\varepsilon,\delta)>0$ such that 
	\begin{equation}\label{sup tilde t}
		\sup_{t\geq0}E_0(t\e)=E_0(\widetilde{t}\e),\quad\text{with}\quad\widetilde{t}\ge t^*.
	\end{equation}

	\begin{lemma}\label{lemma3}
		Assume \ref{F1}-\ref{F2}, \ref{G subcri}-\ref{AR G} and \ref{H1}-\ref{H3}. Suppose also \eqref{sup large}. Then:
		\begin{enumroman}
			\item\label{lemma3-1}  It holds
			\begin{eqnarray*}
				\sup_{t\geq0}E_0(t\e)\leq\begin{cases}
					\dfrac{1}{N}\seq{S_H+c_8\varepsilon^{(N-p)/p}\delta^{-(N-p)/(p-1)}-\lambda c_2\varepsilon^{p-1}}^{N/p}&\quad\text{if}\quad N>p^2,\vspace{0.2cm}\\
					\dfrac{1}{N}\seq{S_H+c_8\varepsilon^{p-1}\delta^{-p}-\lambda c_9\varepsilon^{p-1}|\log(\varepsilon\delta^{-p/(p-1)})|}^{N/p}&\quad\text{if}\quad N=p^2;
				\end{cases}
			\end{eqnarray*}
			\item\label{lemma3-2}   If either $p<N<p^2$ and \ref{G small N nonresonant}, or  if $N^2/(N+1)\leq p^2$ and \ref{G small N} hold we obtain
			\begin{eqnarray*}
				\sup_{t\geq0}E_0(t\e)\leq \dfrac{1}{N}\seq{S_H+c_8\varepsilon^{\frac{N-p}{p}}\delta^{-\frac{N-p}{p-1}}}^{N/p}-c_{10}\varepsilon^{\seq{\frac{p-1}{p}}\seq{N-\frac{q(N-p)}{p}}},%&\text{if}\quad q>p^*\seq{1-\frac{1}{p}}
			\end{eqnarray*}
		\end{enumroman}
		whenever $\delta\in(0,\delta_1)$ and $\varepsilon>0$ satisfy $\varepsilon^\frac{N-p}{p}\delta^{-\frac{N-p}{p-1}}\le 1$, where $c_8,c_9,c_{10}>0$ are constants.
	\end{lemma}

	\begin{proof}
		By the definition of $\widetilde t(\varepsilon,\delta)$ in \eqref{sup tilde t} and $\e$ in \eqref{a}, using \ref{F2} and \ref{H2} again we get
		\begin{eqnarray}\label{g3}
			\sup_{t\geq0}E_0(t\e)%=E_0(\widetilde{t}\e)
			&=&\dfrac{\widetilde{t}^{p}}{p}\seq{1-\lambda p \om F(s_0\widetilde{z}_{\varepsilon,\delta},t_0\widetilde{z}_{\varepsilon,\delta})dx}-\widetilde{t}^{p^*}\om H(s_0\widetilde{z}_{\varepsilon\,\delta},t_0\widetilde{z}_{\varepsilon,\delta})dx-\om G(\widetilde{t}\e)dx\nonumber\\
			&=&\dfrac{\widetilde{t}^{p}}{p}\seq{1-\dfrac{\lambda p F(s_0,t_0)}{\|\nabla z_{\varepsilon,\delta}\|_{p}^p}\om z_{\varepsilon,\delta}^pdx}-\widetilde{t}^{p^*}\dfrac{H(s_0,t_0)}{\|\nabla z_{\varepsilon,\delta}\|_p^{p^*}}-\om G(\widetilde{t}\e)dx\nonumber\\
			&\leq&\dfrac{\tau^p}{p}\seq{\|\nabla z_{\varepsilon,\delta}\|_p^p-c_F\lambda\om z^{p}_{\varepsilon,\delta}dx}-\dfrac{\tau^{p^*}}{p^*}M_H^{p^*/p}-\om G(\widetilde{t}\e)dx
		\end{eqnarray}
		%C^\prime(t^*)^q\om z_{\varepsilon,\delta}^qdx
		for $\delta\in(0,\delta_1)$ and $\varepsilon>0$ such that $\varepsilon^\frac{N-p}{p}\delta^{-\frac{N-p}{p-1}}\le 1$, where $\tau=\widetilde{t}/\|\nabla z_{\varepsilon,\delta}\|_{p}$. Suppose first that $N\geq p^2$. From \eqref{g3} we get 
		\begin{eqnarray}\label{g4}
			E_0(t\e)\leq\dfrac{\tau^p}{p}\seq{\|\nabla z_{\varepsilon,\delta}\|_p^p-c_F\lambda\om z^{p}_{\varepsilon,\delta}dx}-\dfrac{\tau^{p^*}}{p^*}M_H^{p^*/p},
		\end{eqnarray}
		since $G\geq0$. Maximizing the right hand side of \eqref{g4} over all $\tau\geq0$, we get
		\begin{eqnarray*}%\label{g5}
			E_0(t\e)\leq\dfrac{1}{N}\left[M_H^{-1}\seq{\|\nabla z_{\varepsilon,\delta}\|^p_p-c_F\lambda\om z_{\varepsilon,\delta}^pdx}\right]^{N/p}.
		\end{eqnarray*}
		This together with \eqref{estimate1}, \eqref{estimate2} and Lemma \ref{bc3} gives \ref{lemma3-1}.
		Suppose now that either  $p<N<p^2$ and \ref{G small N nonresonant}, or   $N^2/(N+1)\leq p^2$ and \ref{G small N} hold.  In any case we have $G(s,t)\ge c_G(|s|^q+|t|^q)$, with $q>p^*(1/p')$. Since $\widetilde{t}(\varepsilon,\delta)\ge t^*$, by \eqref{estimate1} and \eqref{estimate3}  we get
		\[
		\om G(\widetilde{t}\e)dx\ge \frac{ c_G(|s_0|^q+|t_0|^q) \widetilde{t}^q}{\|\nabla z_{\varepsilon,\delta}\|_p^q}\om z_{\varepsilon,\delta}^q dx
		\ge c_{10}\varepsilon^{\frac{p-1}{p}\seq{N-\frac{q(N-p)}{p}}}.
		\]
		Then,  maximizing the first two terms on the right side of \eqref{g3} over all $\tau\geq0$, we obtain
		\[
		E_0(t\e)\leq\dfrac{1}{N}\left[M_H^{-1}\seq{\|\nabla z_{\varepsilon,\delta}\|^p_p-c_F\lambda\om z_{\varepsilon,\delta}^pdx}\right]^{N/p}-c_{10}\varepsilon^{\frac{p-1}{p}\seq{N-\frac{q(N-p)}{p}}}.
		\]
		Thus, \eqref{estimate1}, \eqref{estimate2} and Lemma \ref{bc3} give \ref{lemma3-2}, which completes the proof of this lemma.
	\end{proof}

	Now, let us consider $\lambda\in(0,\lambda_1)$. In this case, the functional $E_0$ has a mountain pass geometry. 
	\begin{lemma}\label{lemma5}
		Assume \ref{F1}-\ref{F2}, \ref{G subcri}-\ref{AR G} and \ref{H1}-\ref{H3}. If  $\lambda\in(0,\lambda_1)$ then:
		\begin{enumroman}
			\item\label{lemma5-1}  $E_0(0)=0$ and there is $\rho>0$ such that $\inf_{\mathcal{M}_\rho}E_0>0$, where $\mathcal{M}_\rho=\{\rho w:w\in\mathcal{M}\}$;
			\item\label{lemma5-2} If either $N\ge p^2$ or $p<N<p^2$ and \ref{G small N nonresonant} hold, 
			there is $\omega\in W$ such $\|\omega\|>\rho$, $E_0(\omega)<0$ and $\sup_{t\ge0}E_0(t\omega)<\dfrac{1}{N}S_H^{N/p}$.
		\end{enumroman}
	\end{lemma}
	\begin{proof}
		Due to \ref{G subcri} and \ref{H2}, as in \eqref{E(tw) M} but using Theorem \ref{Theorem 4} \ref{Theorem 4.i} in this case, we obtain
		\begin{align*}
			E_0(tw)&\ge\dfrac{{t}^{p}}{p}\seq{1-\frac{\lambda}{\lambda_1}}-{t}^{p^*}\om H(w)dx-\om G({t}w)dx\\
			&\ge \dfrac{{t}^{p}}{p}\seq{1-\frac{\lambda}{\lambda_1}}-C{t}^{p^*}-C_1t^{q_1},\quad\forall w\in\mathcal{M},\,t\in[0,1]
		\end{align*}
		for some $q_1\in(p,p^*)$, due to the boundedness of $\mathcal{M}$. Since $\lambda<\lambda_1$, \ref{lemma5-1} occurs for $\rho>0$ sufficiently small. On the other hand, due to Lemma \ref{lemma3} we see that for a fixed  $\delta\in(0,\delta_1)$ it is possible to choose $\varepsilon>0$ small enough such that
		\[
		\sup_{t\geq0}E_0(t\e)<\dfrac{1}{N}S_H^{N/p}.
		\]
		Considering $\omega=t\e$ for large $t>0$, we get \ref{lemma5-2}.
	\end{proof}

	\subsection{Proofs of main theorems}

	\noindent{\bf Proof of Theorem \ref{Theorem 1}}  
	For $\lambda>0$ we have $0<\lambda<\lambda_1$ or $\lambda_k\leq\lambda<\lambda_{k+1}$ for some $k\in\mathbb{N}$. 
	First we consider $\lambda\in[\lambda_k,\lambda_{k+1})$. In this case we will employ Theorem \ref{lk th}. When $N^2/(N+1)>p^2$, combining \eqref{sup on X} and Lemmas \ref{lemma2-1} and \ref{lemma3}  \ref{lemma3-1} give us
	\begin{eqnarray}\label{g6}
		\sup_{w\in X}E_0(w)\leq c_6\delta^{N(N-p)/p}+ \dfrac{1}{N}\seq{S_H+c_8\varepsilon^{(N-p)/p}\delta^{-(N-p)/(p-1)}-\lambda c_2\varepsilon^{p-1}}^{N/p}.
	\end{eqnarray}
	%with $c_6=0$ in case $\lambda\in(\lambda_k,\lambda_{k+1})$. If $\lambda=\lambda_k$, 
	Set $\delta=\varepsilon^\kappa$, where $\kappa>0$ is yet to be chosen. Then \eqref{g6} yields to
	\[
	\sup_{w\in X}E_0(w)\leq c_6\varepsilon^{\kappa N(N-p)/p}+\frac{1}{N}\seq{S_H+c_8\varepsilon^{(N-p)[1/p-\kappa/(p-1)]}-\lambda c_2\varepsilon^{p-1}}^{N/p}.
	\]
	So the second inequality in \eqref{lk th2} will follow for sufficiently small $\varepsilon>0$ if $\kappa>0$ can be chosen so that $\min\{\kappa N(N-p)/p;(N-p)[1/p-\kappa/(p-1)]\}> p-1$. This is only possible if
	\[
	(p-1)\left[\frac{1}{p}-\frac{p-1}{N-p}\right]>\frac{p(p-1)}{N(N-p)},
	\]
	which is equivalent to $N^2/(N+1)>p^2$. Joining this information with Proposition \ref{bc4}, \eqref{same index} and \eqref{supA<infB}, we can apply Theorem \ref{lk th} to get a nontrivial critical point of $E_0$.
	
	Suppose now that $\lambda\in[\lambda_k,\lambda_{k+1})$, $N^2/(N+1)\leq p^2$ and \ref{G small N} hold. Then \eqref{sup on X} and Lemmas \ref{lemma2-1}, \ref{lemma3} \ref{lemma3-2} combined give us
	\begin{eqnarray}\label{g7}
		\sup_{w\in X}E_0(w)\leq c_6\delta^{N(N-p)/p}+\dfrac{1}{N}\seq{S_H+c_8\varepsilon^{\frac{N-p}{p}}\delta^{-\frac{N-p}{p-1}}}^{N/p}-c_{10}\varepsilon^{\frac{p-1}{p}\seq{N-\frac{q(N-p)}{p}}}.
	\end{eqnarray}
	Set $\delta=\varepsilon^{\kappa}$, where $\kappa>0$ is yet to be chosen. Then \eqref{g7} yields to
	\[
	\sup E_0(X)\leq c_6\varepsilon^{\kappa N(N-p)/p}+\frac{1}{N}\seq{S_H+c_8\varepsilon^{(N-p)\left(\frac{1}{p}-\frac{\kappa}{p-1}\right)}}^{N/p}-c_{10}\varepsilon^{\frac{p-1}{p}\seq{N-\frac{q(N-p)}{p}}}.
	\]
	So the second inequality in \eqref{lk th2} will follow for sufficiently small $\varepsilon>0$ if $\kappa>0$ can be chosen so that 
	$$
	\min\left\{\kappa \frac{N(N-p)}p; (N-p)\seq{\frac1p-\frac{\kappa}{p-1}}\right\}>\frac{p-1}{p}\seq{N-\frac{q(N-p)}p}.
	$$
	This is possible if and only if 
	\begin{eqnarray*}
		\dfrac{(p-1)(Np-q(N-p))}{pN(N-p)}<\dfrac{p-1}{p}\seq{\dfrac{2N-Np-p}{N-p}}+q\dfrac{(p-1)^2}{p^2},
	\end{eqnarray*}
	which is the same as the condition on $q$ given in \ref{G small N},
	\(
	q>p^*-{Np}/[{N(p-1)+p}].
	\)
	As before,   Theorem \ref{lk th} now gives us a nontrivial critical point of $E_0$.
	
	If $\lambda\in(0,\lambda_1)$, due to Lemma \ref{lemma5} and Proposition \ref{bc4} we can apply the Mountain Pass theorem of Ambrosetti and Rabinowitz \cite{AR73} to get a nontrivial critical point for $E_0$. 
	This completes the proof of Theorem \ref{Theorem 1}.
	\bigskip
	
	\noindent{\bf Proof of Theorem \ref{Theorem 2}}.
	As in the proof of Theorem \ref{Theorem 1}, when $\lambda\in(0,\lambda_1)$,  the Mountain Pass theorem can be applied to provide a nontrivial critical point for $E_0$.  Let us consider $\lambda\in(\lambda_k,\lambda_{k+1})$. In this case we will employ again Theorem \ref{lk th}.  We fix $\delta\in(0,\delta_1)$ small enough so that $\lambda_k+c_3\delta^{N-p}\leq\lambda$. Therefore, by \eqref{sup on X} and Lemma \ref{lemma2-1} one gets
	\begin{equation}\label{sup X not eigenvalue}
		\sup_{w\in X}E_0(w)\leq  \sup_{t\geq0}E_0(t\e).
	\end{equation}
	If $N\geq p^2$,  due to Lemma \ref{lemma3} \ref{lemma3-1} we obtain
	\begin{eqnarray*}
		\sup_{w\in X}E_0(w)\leq \begin{cases}
			\dfrac{1}{N}\seq{S_H+c_8\varepsilon^{(N-p)/p}\delta^{-(N-p)/(p-1)}-\lambda c_2\varepsilon^{p-1}}^{N/p}&\quad\text{if}\quad N>p^2,\vspace{0.2cm}\\
			\dfrac{1}{N}\seq{S_H+c_8\varepsilon^{p-1}\delta^{-p}-\lambda c_9\varepsilon^{p-1}|\log(\varepsilon\delta^{-p/(p-1)})|}^{N/p}&\quad\text{if}\quad N=p^2.
		\end{cases}
	\end{eqnarray*}
	Choosing $\varepsilon>0$ small enough we verify that
	\begin{align}\label{sup E0 X c*}
		\sup_{w\in X}E_0(w)<\dfrac{1}{N}S_H^{N/p}.
	\end{align}
	Joining this information with Proposition \ref{bc4}, \eqref{same index} and \eqref{supA<infB}, we see that Theorem \ref{lk th} can be applied and gives us a nontrivial solution to system \eqref{101}.
	
	Suppose now that $p<N<p^2$ and \ref{G small N nonresonant} hold. Then  by \eqref{sup X not eigenvalue} and Lemma \ref{lemma3} \ref{lemma3-2} we have
	\begin{eqnarray*}
		\sup_{w\in X}E_0(w)\leq \dfrac{1}{N}\seq{S_H+c_8\varepsilon^{\frac{N-p}{p}}\delta^{-\frac{N-p}{p-1}}}^{N/p}
		-c_{10}\varepsilon^{\frac{p-1}{p}\seq{N-\frac{q(N-p)}{p}}}.
	\end{eqnarray*}
	Due to the range allowed for $q$ in \ref{G small N nonresonant}, %$(N-p)/p > [(p-1)/p][N-q(N-p)/p]$ if and only if $q>p^*-p'$
	we can choose $\varepsilon>0$ small enough so that \eqref{sup E0 X c*} also holds in this case.
	As before, Theorem \ref{lk th}  provides a nontrivial solution to system \eqref{101}.
	This concludes the proof of Theorem \ref{Theorem 2}.

	\bigskip
	
	\noindent  {\bf Proof of Theorem \ref{Theorem 5}}. In case $\lambda\ge\lambda_1$ we will employ Theorem \ref{thm:2.1}, with $C_\delta$ given in \eqref{C_delta def}, %where $\delta>0$ is small enough so that $\lambda_k+c_3\delta^{N-p}\leq\lambda$, 
	$w_\delta=\e$ defined in \eqref{a} and $R>0$ as given by Lemma \ref{lemma1-5}. This same lemma gives us inequality \eqref{eq:2.8}.  Under the assumptions of Theorem \ref{Theorem 1} or \ref{Theorem 2}, we verified in their proofs that it is possible to choose $\varepsilon$ and $\delta>0$ sufficiently small so that
	\[
	\sup_{w\in X}E_0(w)<\frac{1}{N}S_H^{N/p}
	\]
	for $X$ defined in \eqref{X def}, which can be written as $X=\{tw: w\in C_\delta\cup A_\delta,\,t\in[0,R]\}$.  This implies that inequality \eqref{eq:2.9} holds. Therefore, Theorem \ref{thm:2.1} ensures the existence  of $\mu_0>0$ such that $E$ has two nontrivial critical points,   whence system 
	\eqref{101 pert} has two nontrivial solutions, whenever $0<\|\gamma\|_{(p^*)'}+\|\eta\|_{(p^*)'}<\mu_0$.
	
	\bigskip
	
	In the case \( \lambda \in (0,\lambda_{1}) \), the variational geometry becomes
	considerably simpler.
	Indeed, it suffices to choose \( w_{\delta} = e_{\varepsilon,\delta} \) as defined in
	\eqref{a}.
	Using the estimates obtained above (Proposition \ref{bc4} and Lemma \ref{lemma5}), one verifies that condition~\eqref{eq:2.91}
	is satisfied for \( \varepsilon, \delta > 0 \) sufficiently small. Consequently,  the assumptions of Theorem~\ref{thm:2.1 small lambda} are fulfilled,
	and the conclusion follows.
	%%%%%%%%%%%%%%%%%%%%%%%%%%%%%%%%%%%%%%%%%%%%%%%%%%%%%%%%%%%%%%%%%%%%%%%%%%%%%%%%%%%%%%%%%%%%%%%%%%%%%%%%%%%%%%%%%%%%
	
	\bigskip
	
	\noindent\textbf{Acknowledgments.}
	B.~Ribeiro acknowledges financial support from CNPq-Brazil through grants
	443594/2023-6, 314111/2023-9, and 201452/2024-3.
	E.~Gloss acknowledges financial support from CNPq through grants
	201454/2024-6 and 443594/2023-6.
	
	\bigskip
	
	\noindent\textbf{Data availability statement:}
	This manuscript does not use any data.
	
	\medskip
	
	\noindent\textbf{Conflict of interest statement:}
	The authors declare that there are no conflicts of interest.

	%%%%%%%%%%%%%%%%%%%%%%%%%%%%%%%%%%%%%%%%%%%%%%%%%%%%%%%%%%%%%%%%%%%%%%%%%%%%%%%%%%%%%%%%%%%%%%%%%%%%%%%%%%%%%%%%%%%%

\end{document}